\documentclass[11pt,reqno]{amsart}

\usepackage[T1]{fontenc}
\usepackage[utf8]{inputenc}
\usepackage{amsmath,amssymb}
\usepackage{enumitem}
\usepackage[hidelinks,bookmarks=false]{hyperref}
\usepackage[nameinlink,noabbrev]{cleveref}

\theoremstyle{plain}
\newtheorem{theorem}{Theorem}[section]
\newtheorem{lemma}[theorem]{Lemma}
\newtheorem{proposition}[theorem]{Proposition}

\theoremstyle{definition}
\newtheorem{definition}[theorem]{Definition}
\newtheorem{remark}[theorem]{Remark}

\makeatletter
\let\amsremark\remark
\let\endamsremark\endremark
\newif\ifremarkendplaced
\newcommand{\remarkendsymbol}{\ensuremath{\clubsuit}}
\newcommand{\remarkendmark}{%
  \leavevmode\unskip\nobreak\hbox{}\nobreak\hfill
  \makebox[0pt][r]{\quad\hbox{\remarkendsymbol}}%
  \global\remarkendplacedtrue
}
\newcommand{\remarkendhere}{%
  \ifmmode
    \global\tag@true
    \global\@eqnswfalse
    \gdef\df@tag{%
      \maketag@@@{\remarkendsymbol\quad\tagform@\theequation}%
      \def\@currentcounter{equation}%
      \def\@currentlabel{\p@equation\theequation}%
    }%
    \global\remarkendplacedtrue
  \else
    \remarkendmark
  \fi
}
\renewenvironment{remark}[1][]{%
  \global\remarkendplacedfalse
  \@ifempty{#1}{\amsremark}{\amsremark[#1]}%
}{%
  \ifremarkendplaced\else\remarkendmark\fi
  \endamsremark
}

\renewenvironment{proof}[1][\proofname]{\par
  \pushQED{\qed}%
  \normalfont \topsep6\p@\@plus6\p@\relax
  \trivlist
  \item[\hskip\labelsep
        \bfseries\itshape
    #1\@addpunct{.}]\ignorespaces
}{%
  \popQED\endtrivlist\@endpefalse
}
\let\amsparagraph\paragraph
\renewcommand{\paragraph}[1]{\amsparagraph{\underline{#1}}}

\newcommand{\R}{\mathbf R}
\newcommand{\bfS}{\mathbf S}
\newcommand{\cA}{\mathcal{A}}

\newcommand{\cI}{\mathcal{I}}
\newcommand{\cX}{\mathcal{X}}
\newcommand{\e}{\mathrm{e}}
\newcommand{\ud}{\,\mathrm{d}}
\newcommand{\1}{\mathbf 1}
\newcommand{\half}{\frac12}
\let\tilde\widetilde
\renewcommand{\le}{\leqslant}

\renewcommand{\leq}{\leqslant}
\renewcommand{\geq}{\geqslant}
\newcommand{\argmin}{\mathop{\rm arg\,min}}
\newcommand{\argmax}{\mathop{\rm arg\,max}}
\DeclareMathOperator{\conv}{\bf conv}
\DeclareMathOperator{\cl}{\bf cl}
\DeclareMathOperator{\cone}{\bf cone}
\newcommand{\Var}{\mathrm{Var}}
\newcommand{\E}{\operatorname*{\mathbf{E}}\ilimits@}
\renewcommand{\P}{\operatorname*{\mathbf{P}}\ilimits@}
\newcommand{\ie}{\textit{i}.\textit{e}., }
\newcommand{\eg}{\textit{e}.\textit{g}., }
\newcommand{\cprime}{\raise.9ex\hbox{\scriptsize$\prime$}}
\makeatother

\let\asymp\simeq
\newcommand{\twomax}[2]{\ensuremath{#1 \vee #2}}
\newcommand{\twomin}[2]{\ensuremath{#1 \wedge #2}}
\newcommand{\Proj}[1]{\Pi_{#1}}
\newcommand{\rad}{\mathrm{rad}}
\newcommand{\CvxLip}[1]{\mathsf{C}_{\rm CL}(#1)}
\newcommand{\RightDer}{\mathsf{D}_+}
\newcommand{\AllConvSets}[1]{\mathcal{C}_{#1}}
\newcommand{\AllConvCones}[1]{\mathcal{K}_{#1}}
\newcommand{\ExtProbProj}[3]{\mathsf{A}^\star_{#1, #2}(#3)}
\newcommand{\ExtProbProjCone}[3]{\tilde{\mathsf{A}}^\star_{#1, #2}(#3)}
\newcommand{\ExtProbSuppVec}[3]{\mathsf{B}^\star_{#1, #2}(#3)}
\newcommand{\SuppVec}[1]{s_{#1}}

\title[Metric projections onto convex sets]{On the Metric Projection onto a Convex Set: \\[0.1em]
\small reverse H\"older inequalities and upper bounds}
\author{Reese Pathak}
\thanks{We gratefully acknowledge support from the NSF under grant
DMS-2503579}
\address{Department of Statistics, UC
Berkeley,\and\newline\hspace*{\parindent}%
School of Oper. Res. \& Inf. Engineering (ORIE), Cornell University}
\email{pathakr@berkeley.edu}

\date{\today}

\begin{document}

\begin{abstract}
We study the $L^p(\mu)$-norm of the metric projection onto 
a closed, convex set $C \subset \R^n$ when $\mu$ is
the uniform measure on the sphere 
or the standard Gaussian measure on $\R^n$. 
Up to universal constants, we determine the optimal reverse Hölder
inequalities (\ie $L^q-L^p$ estimates for $q > p$) for both settings and for
all
$1 \leq p < q \leq \infty$. The optimal constants in these inequalities depend polynomially on the \mbox{dimension $n$}. 
We establish upper bounds for the expected norm of the metric projection for a wide class of probability measures. Our inequalities
improve and extend previous results of S. Chatterjee. 
\end{abstract}

\maketitle

\section{Introduction}

Let $\AllConvSets{n}$ denote the class of closed, convex sets $C \subset \R^n$ 
that contain the origin. The metric projection onto $C \in \AllConvSets{n}$ at the point 
$x \in \R^n$ is given by 
\[
\Proj{C}(x) = \argmin_{y \in C} \, \|x - y\|_2.
\]
The key quantity of interest in this paper is the \emph{norm} of the
projection, $\|\Proj{C}(x)\|_2$, and, 
in particular, its behavior under (standard) Gaussian measure $\gamma_n$ on
$\R^n$ and the uniform measure $\sigma_{n-1}$ on the unit sphere $\bfS^{n-1}$. 

The norm of the metric projection plays a key role in several questions in both pure and applied
mathematics. For instance, by a recent result of the author and N.\
Zhivotovskiy~\cite[Theorem 2.2]{PatZhi26}, for bounded $C \in \AllConvSets{n}$,
\[
\E \bigg[\sup_{x \in C} \, \langle G, x\rangle \bigg]=
\half \int_0^\infty \frac{\E\|\Proj{\lambda C}(G)\|_2^2}{\lambda^2}\, \ud
\lambda,
\]
where the expectations above are taken with respect to $G \sim \gamma_n$.
Hence, the Gaussian width---a key quantity in
geometric
analysis~\cite{ArtAviGiaMil15}---is closely related to the typical size of
projections
onto
dilates of $C \in \AllConvSets{n}$,
\[
\lambda C = \{\lambda x: x \in C\}, \quad \mbox{for}~\lambda > 0.
\]

In applied mathematics---particularly in statistical estimation, signal
processing, and optimal recovery---the norm of the projection governs the error
in certain recovery procedures. Consider a noisy observation
\[
Y = x_\star + \xi, \quad \mbox{for some}~x_\star \in \cX \subset \R^n,
\] 
where $\xi$ is some random vector (``noise''). If $\cX$ is closed and
convex, it is common to recover $x_\star$ via 
\[
\hat x \equiv \hat x(Y) = \argmin_{x \in \cX} \, 
\|Y-x\|_2.
\]
Depending on the context, this method of signal recovery is
known as maximum likelihood estimation, least squares estimation, or empirical
risk minimization; we refer to the books~\cite{vdG00,Tsy09,Wai19}.
The recovery error satisfies:
\[
\E \|\hat x(Y) - x_\star\|_2^p = 
\E \|\Pi_{\cX - x_\star}(\xi)\|_2^p.
\]
Notably, the translate $\cX-x_\star$ is an element of
$\AllConvSets{n}$;
plainly,
it
is
a
closed,
convex set that contains the origin. As seen above, the stochastic behavior of
the
norm
of the projection onto such sets essentially determines the recovery error.
This was also the motivation of the seminal paper of S.\ Chatterjee~\cite{Cha14},
which we discuss in greater detail in~\Cref{sec:bounds-on-the-expected-norm}.

\medskip 

We now turn to our main results.

\subsection{Failure of dimension-free reverse Hölder inequalities.} 

Let $\xi \sim \mu$ be a random
vector
in
$\R^n$ and let $C \in \AllConvSets{n}$. A basic question regarding metric projections is whether they enjoy good \emph{reverse Hölder inequalities}: \ie whether the $L^q(\mu)$ and $L^p(\mu)$ norms of
$\|\Pi_C(\xi)\|_2$ are
comparable,
for
$q$
not much larger than $p$. 

A classical setting where these inequalities hold is when
$\mu$ is log-concave and $f \colon \R^n \to \R_+$ is a seminorm.
Then (\eg~\cite[Theorem~3.5.11]{ArtAviGiaMil15}):
\[
\|f\|_{L^q(\mu)} \lesssim \frac{q}{p} \, 
\|f\|_{L^p(\mu)}, 
\quad \mbox{for}~1 \leq p \leq q < \infty.
\]
The inequality is referred to as \emph{dimension-free} because it does not depend on the ambient dimension $n$.
The map ${x
\mapsto
\|\Pi_C(x)\|_2}$, for general $C \in \AllConvSets{n}$, does not 
satisfy these assumptions: it is \emph{not} convex, even, or $1$-positively homogeneous. 
Thus,
it is \emph{a priori} unclear
whether such dimension-free inequalities can hold.
Our results show that such reverse Hölder inequalities
\emph{cannot} be dimension-free, either on the sphere or in Gauss space.

\medskip 

Formally, we seek the smallest number $A \geq 1$---as a function of $n, p,
q$---for
which the reverse Hölder inequality, 
\[
\Big(\E_\mu \|\Proj{C}(X)\|_2^q\Big)^{1/q} \leq A \, \Big(\E_\mu
\|\Proj{C}(X)\|_2^p\Big)^{1/p},
\]
holds uniformly for $C \in \AllConvSets{n}$, for a given pair $(p,q)$ with $1 \leq
p
\leq q \leq \infty$. 
Equivalently, we study the extremal problem
\[
\ExtProbProj{p}{q}{\mu}
 = \sup_{C \in \AllConvSets{n}, C \neq \{0\}} \,
\frac{\|\Proj{C}\|_{L^q(\mu)}}{\|\Proj{C}\|_{L^p(\mu)}}.
\]
Note that we always restrict to $C \in \AllConvSets{n}$ for which the denominator above is positive. Additionally, note that for a vector-valued
measurable map $f
\colon
\R^n \to \R^m$, we set 
\[
\|f\|_{L^p(\mu)} = \Big(\E_{\xi\sim \mu}
\|f(\xi)\|_{2}^p\Big)^{1/p} \quad \mbox{for}~1 \leq p < \infty.
\]
We set $\|f\|_{L^\infty(\mu)}$ to be the $\mu$-essential supremum of $\|f(\xi)\|_2$. 
\medskip 

As stated next, we determine $\ExtProbProj{p}{q}{\mu}$, apart from universal constants, for spherical measure\footnote{By a rescaling argument, \Cref{thm:optimal-reverse-holder-sphere} also holds for $\mu = \mathrm{Unif}(\alpha \bfS^{n-1})$, for any $\alpha > 0$.},
\ie $\mu =
\sigma_{n-1}$ (\Cref{thm:optimal-reverse-holder-sphere}), and for Gauss space,
\ie $\mu =
\gamma_n$ (\Cref{thm:optimal-reverse-holder-Gauss}).  

\medskip 

Throughout, for two
functions
$f, g \colon \cI \to \R_+$ we write $f \lesssim g$ (or $g \gtrsim f$) if there exists a constant $C > 0$ such that
$f(\iota)
\leq C\, g(\iota)$ for all $\iota \in \cI$. If $f\lesssim g$ and $g\lesssim f$, we
write $f \simeq g$. For $r \in [1, \infty)$,
we
put 
\[
\twomin{r}{n} = \min\{r, n\}, \quad \mbox{and} \quad
\twomax{r}{n} = \max\{r, n\}.
\]
For $r= \infty$, we take $\twomin{r}{n} = n$. Put $p/q = 0$ for $p$ finite
and $q = \infty$. 

\begin{theorem}[Solution to extremal problem on the sphere]
\label{thm:optimal-reverse-holder-sphere}
For $n \geq 1$ and $1 \leq p < q \leq \infty$, it holds that
\[
\ExtProbProj{p}{q}{\sigma_{n-1}}\simeq
\bigg(\frac{n}{(\twomin{p}{n})\log(\e \tfrac{n}{\twomin{q}{n}})}\bigg)^\tau,
\quad \mbox{where}~~\tau = \frac{1}{2}\Big(1 -
\frac{p}{q}\Big).
\]
\end{theorem}

\Cref{thm:optimal-reverse-holder-sphere} may be easier to interpret when rewritten equivalently as:
\[
\ExtProbProj{p}{q}{\sigma_{n-1}}
\asymp 
\begin{cases}
     \big(\frac{n}{p \log (\e n/p)}\big)^\tau, & \mbox{if}~~p \in [1, n]~\mbox{and}~q \in (p, p \log \tfrac{\e n}{p}]\\[0.5em] 
     \sqrt{\frac{n}{p \log (\e n/q)}}, & \mbox{if}~p \in [1, n]~\mbox{and}~q \in (p \log \tfrac{\e n}{p}, n] \\  
\sqrt{\frac{n}{p}}, & \mbox{if}~p \in [1, n]~\mbox{and}~q > n \\ 
     1, & \mbox{if}~p \geq n
\end{cases}.
\]

We now present our main result in the Gaussian setting. By taking $C = \R^n$ (or by considering $C = rB^n_2$ with
$r \to \infty$), we see that
no
finite
estimate
can
hold when $q=\infty$ and $p<\infty$. Otherwise, our next result characterizes
the
optimal
reverse Hölder inequalities for projections onto closed convex sets under
Gaussian measure.

\begin{theorem}[Solution to extremal problem in Gauss space]
\label{thm:optimal-reverse-holder-Gauss}
For $n \geq 1$ and $1 \leq p < q < \infty$, it holds that
\[
\ExtProbProj{p}{q}{\gamma_n}\simeq
\sqrt{\frac{\twomax{q}{n}}{\twomax{p}{n}}}
\bigg(\frac{n}{(\twomin{p}{n})\log(\e \tfrac{n}{\twomin{q}{n}})}\bigg)^\tau,
\quad \mbox{where}~~\tau = \frac{1}{2}\Big(1 -
\frac{p}{q}\Big).
\]
\end{theorem}

\Cref{thm:optimal-reverse-holder-Gauss} may be easier to interpret when rewritten equivalently as:
\[
\ExtProbProj{p}{q}{\gamma_n}
\asymp 
\begin{cases}
    (\frac{n}{p \log (\e n/p)})^\tau, & \mbox{if}~~p \in [1, n]~\mbox{and}~q \in (p, p \log \tfrac{\e n}{p}]\\[0.5em] 
     \sqrt{\frac{n}{p \log (\e n/q)}}, & \mbox{if}~p \in [1, n]~\mbox{and}~q \in (p \log \tfrac{\e n}{p}, n] \\  
\sqrt{\frac{q}{p}}, & \mbox{if}~q > n
\end{cases}.
\]

\medskip 

We emphasize that even for $q$ near
$p$, the optimal reverse Hölder inequality must scale polynomially in the
dimension. Indeed, taking $p = 1$ and $q = 2$, 
by~\Cref{thm:optimal-reverse-holder-sphere,thm:optimal-reverse-holder-Gauss},
there
are
convex
sets
$C \in \AllConvSets{n}$ such that 
\[
\E \|\Proj{C}(\xi)\|_2^2 \asymp \sqrt{\frac{n}{1 + \log n}}
\Big(\E
\|\Proj{C}(\xi)\|_2\Big)^2.
\]
Above, $\xi$ may be drawn either uniformly from the sphere
$\bfS^{n-1}$
or according to the standard Gaussian distribution. 
Note, however, that if we replace the class of sets $\AllConvSets{n}$ by the set of all closed convex \emph{cones} in $\R^n$, 
then dimension-free reverse Hölder inequalities \emph{are} possible; 
see \Cref{remark:cone} and~\Cref{thm:extremal-problem-cones}.
We defer further discussion of~\Cref{thm:optimal-reverse-holder-sphere,thm:optimal-reverse-holder-Gauss} to
\Cref{sec:additional-remarks}, after presenting the proofs below. 

\subsection{Bounds on the expected
norm of metric projections}
\label{sec:bounds-on-the-expected-norm}

Let $\xi \sim \mu$ be a random vector in $\R^n$ and fix a convex set 
$C \in \AllConvSets{n}$. Our focus in this 
section is on upper bounds on the quantity
\[
\|\Proj{C}\|_{L^p(\mu)} = \Big(\E \|\Proj{C}(\xi)\|_{2}^p\Big)^{1/p},
\]
which depend on the pair $(C, \mu)$. Our primary focus is the expected norm itself (\ie $p = 1$), but we comment on some generalizations for $p \geq 1$. 

We begin by characterizing the qualitative behavior of the norm of the projection.
In order to state the result, we need to introduce the concept of \emph{minimal norm supporting vectors}.
For $C \in \AllConvSets{n}$, recall that the support function of $C$ at $x \in \R^n$ is given by 
\[
h_C(x) = \sup_{y \in C} \, \langle x, y \rangle. 
\]
\begin{definition}[Minimal norm supporting vectors]
\label{def:minimal-norm-supp-vec}
For a compact convex set $C \subset \R^n$, the \emph{minimal norm supporting vector in $C$ at $x$} is
given by 
\[
\SuppVec{C}(x) = \argmin_{y \in C} \Big\{\, \|y\|_2 : \, ~\langle x, y\rangle = h_C(x)\, \Big\}.
\]
\end{definition}

These supporting vectors play a key role in the analysis of the expected norm of the metric projection. In the result below, $\cl A$ denotes the closure and $\cone A = \R_+ A$ denotes the conic hull of a set $A \subset \R^n$.

\begin{proposition}[Qualitative behavior of the metric projection]
\label{prop:qualitative-behavior-of-metric-projections}
For every compact $C \in \AllConvSets{n}$ and $p \in [1, \infty]$, the following hold. 
\begin{enumerate}[label=(\roman*)]
\item 
\label{item:cone-upper-bound-and-limit}
If $\|\Proj{\cl \mathbf{cone}~C}\|_{L^p(\mu)} < \infty$, then
\[
\|\Proj{\lambda C}\|_{L^p(\mu)} \leq \|\Proj{\cl \cone C}\|_{L^p(\mu)}, \quad \mbox{for each $\lambda > 0$}.
\]
 Moreover, $
\lim_{\lambda \to \infty} \|\Proj{\lambda C}\|_{L^p(\mu)} =\|\Proj{\cl \cone C}\|_{L^p(\mu)}$. 
\medskip
\item
\label{item:supporting-vector-upper-bound-and-limit} 
For every $\lambda > 0$, 
$\|\Proj{\lambda C}\|_{L^p(\mu)} \leq \lambda \|\SuppVec{C}\|_{L^p(\mu)}$. Moreover, 
\[
\lim_{\lambda \to 0^+} \frac{\|\Proj{\lambda C}\|_{L^p(\mu)}}{\lambda} =
\|\SuppVec{C}\|_{L^p(\mu)}.
\] 
\end{enumerate}
\end{proposition}
 
\begin{definition}[Typical radius]
Given a random vector $\xi \sim \mu$ in $\R^n$ and a convex set $C \in \AllConvSets{n}$, 
the \emph{typical radius} is defined by 
\[
r_\mu(C) = \argmax_{r \geq 0} \bigg\{\, 
\E_\mu h_{C \cap rB^n_2}(\xi) - \frac{r^2}{2} \,\bigg\}.
\]
\end{definition}

Chatterjee~\cite{Cha14} related the tails of the projection to the radius $r_\mu(C)$ in the Gaussian setting $\mu = \gamma_n$ by leveraging concentration of measure in Gauss space.
Our next result establishes a broad generalization of this result: under a minimal assumption on the pair $(C, \mu)$, the typical radius still bounds the typical size of the metric projection.

\begin{proposition}
\label{prop:upper-bound-via-typical-radius}
Fix a set $C \in \AllConvSets{n}$. Then, the following hold.
\begin{enumerate}[label=(\roman*)]
\item     
\label{item:well-defined-fixed-point}
The typical radius $r_\mu(C)$ is well-defined (\ie unique and finite) if and only if $\E_\mu h_{C \cap sB^n_2}(\xi) < \infty$ for some $s > 0$. 
\item 
\label{item:upper-bound-on-mean-norm}
Whenever $r_\mu(C)$ is well-defined, it holds that 
\[
\|\Proj{C}\|_{L^1(\mu)} \leq 2\, r_\mu(C).
\]
\end{enumerate}
\end{proposition}

Next, we control the gap between the typical radius and the mean norm of the projection, when the variance of convex, Lipschitz functionals is bounded.

\begin{definition}
\label{def:convex-lipschitz-constant}
Let $\mu$ be a probability measure on $\R^n$. Define the \emph{convex Lipschitz constant of $\mu$} by 
\[
\CvxLip{\mu} = \sup\Big\{\, \sqrt{\Var_\mu(f)} \mid  f \colon \R^n \to \R, ~\text{convex, $1$-Lipschitz}\, \Big\}.
\]
\end{definition}

In~\Cref{def:convex-lipschitz-constant}, the maps are assumed 
$1$-Lipschitz with respect to the Euclidean norm $\|\cdot\|_2$. 
Note that $\CvxLip{\mu} < \infty$ if and only if 
$\E_\mu \|\xi\|_2^2 < \infty$; in particular $\CvxLip{\mu} \lesssim \sqrt{\E_\mu[\|\xi\|_2^2]}$. These quantities can differ in order, \eg when $\mu$ is the uniform measure on the hypercube $\{-1,1\}^n$: 
\[
\CvxLip{\mu} \asymp 1 \ll \sqrt{n} = (\E_\mu \|\xi\|_2^2)^{1/2}.
\]
Above, the first relation follows from Talagrand's convex distance inequality~\cite[Theorem 4.1.1]{Tal95}.

\begin{theorem}
\label{thm:distance-from-typical-radius}
Suppose that $\mu$ is a probability measure on $\R^n$ such that $\CvxLip{\mu} < \infty$. Then, 
\begin{align*}
\Big|\|\Proj{C}\|_{L^1(\mu)} - r_\mu(C) \Big|
&\leq 
\E\Big| \|\Proj{C}(\xi)\|_2 - r_\mu(C)\Big| \\  
&\lesssim 
\min\bigg\{r_\mu(C), \sqrt{\CvxLip{\mu} \, r_\mu(C)}\,\bigg\},
\end{align*}
for every $C \in \AllConvSets{n}$. 
\end{theorem}

While the mapping $x \mapsto \|\Proj{C}(x)\|_2$ is $1$-Lipschitz for $C \in \AllConvSets{n}$, it is generally not convex. Nonetheless, we show that it suffices to control the variance of \emph{convex}, Lipschitz maps, rather than \emph{all} Lipschitz maps. The underlying idea is that for $C \in \AllConvSets{n}$, the projection is essentially determined by
\[
x \mapsto h_{C \cap rB^n_2}(x), \quad \mbox{for}~r \geq 0.
\]
In the proof of~\Cref{thm:distance-from-typical-radius}, we establish a more general result (\Cref{lem:tail-bounds-for-projection-norm}), where uniform boundedness of the variance is replaced by uniform boundedness of an Orlicz norm over centered, convex, $1$-Lipschitz functionals. 

\begin{remark}[An improvement of~\Cref{thm:distance-from-typical-radius}]
One may replace $\CvxLip{\mu}$ by the smaller quantity 
\begin{multline*}
\tilde{\CvxLip{\mu}} = \sup\Big\{\sqrt{\Var_\mu(f)} \mid f \colon \R^n \to \R,\text{\small convex}\\\text{\small $1$-Lipschitz, $1$-positively homogeneous}\Big\}.
\end{multline*}
This only improves constants, however, since $\tilde{\CvxLip{\mu}} \asymp \CvxLip{\mu}$, with 
implicit constants independent of the measure $\mu$. 
\end{remark}

\begin{remark}[Sharpness of~\Cref{thm:distance-from-typical-radius}]
The following example shows that \Cref{thm:distance-from-typical-radius}, up to universal constants, cannot be improved uniformly. Put $C(\tau) = \{\lambda e_1 : 0 \leq \lambda \leq \tau\}$ for some $\tau \in (0, 1/2)$. Let $\mu$ denote the uniform measure on the hypercube $\{-1, 1\}^n$. By direct computation, $r_{\mu}(C(\tau))= \tau$, and $\|\Proj{C(\tau)}(x)\|_2 = \tau \1\{x_1 = 1\}$ for $x \in \{-1, 1\}^n$. Hence, 
\begin{multline*}
\Big| \|\Proj{C(\tau)}\|_{L^1(\mu)} - r_\mu(C(\tau))\Big| \\ = \frac{\tau}{2} \asymp
r_\mu(C(\tau)) \asymp  
\min\Big\{r_\mu(C(\tau)), \sqrt{\CvxLip{\mu} r_\mu(C(\tau))}\Big\},
\end{multline*}
as needed. 
\end{remark}

\begin{remark}[$L^p$ estimates for the projection]
We now discuss some bounds on the $L^p$-norm of the projection that one can derive from our results. For simplicity, we focus on the case where $\mu = \gamma_n$; similar results (with the obvious modifications) are possible for the uniform measure on the sphere, or more generally when $\mu$ has bounded Poincaré constant. 

In the Gaussian setting, the Borell-TIS inequality~\cite{SudTsi74,Bor75} implies that
\[
\Big\|\|\Proj{C}\|_2 - \E \|\Proj{C}\|_2\Big\|_{\psi_2(\gamma_n)} \lesssim 1, 
\]
for all $n \geq 1$, and all $C \in \AllConvSets{n}$. 
Equivalently, for all $p \geq 1$, it holds that 
\[
\begin{gathered}
\Big\|\|\Proj{C}\|_2 - \E \|\Proj{C}\|_2\Big\|_{L^p(\gamma_n)} \leq A \sqrt{p}, \quad \mbox{and, hence,}\\ \|\Proj{C}\|_{L^p(\gamma_n)} \leq \|\Proj{C}\|_{L^1(\gamma_n)}
+ A\sqrt{p}.
\end{gathered}
\]
In particular, from~\Cref{thm:distance-from-typical-radius}, since $\CvxLip{\gamma_n} = 1$, we also obtain
\begin{align}
    \Big\|\|\Proj{C}\|_2 &- r_{\gamma_n}(C)\Big\|_{L^p(\gamma_n)} 
    \nonumber \\&\leq 
    \Big\|\|\Proj{C}\|_2 - \E \|\Proj{C}\|_2\Big\|_{L^p(\gamma_n)} + 
    \Big|\|\Proj{C}\|_{L^1(\gamma_n)} - r_{\gamma_n}(C) \Big| 
    \nonumber \\ 
    &\leq A \sqrt{p} + B \min\{r_{\gamma_n}(C), \sqrt{r_{\gamma_n}(C)}\}, \quad \mbox{for all}~p\geq 1.
\label{ineq:Lp-estimates}
\end{align}
In the above discussion, $A, B > 0$ are universal constants. 
\end{remark}

\begin{remark}[Comparison with~\cite{Cha14}]
Chatterjee's paper~\cite{Cha14} focused on the Gaussian setting (\ie $\mu = \gamma_n$). He showed that for any $C \in \AllConvSets{n}$, 
\[
\Big|\|\Proj{C}\|_{L^1(\gamma_n)} - r_{\gamma_n}(C) \Big| \lesssim 
\max\Big\{1, \sqrt{r_{\gamma_n}(C)}\Big\}.
\]
This inequality produces a larger estimate (when $r_{\mu}(C) \ll 1$) than our \Cref{thm:distance-from-typical-radius}, which implies 
\[
\Big|
\|\Proj{C}\|_{L^1(\gamma_n)} 
- r_{\gamma_n}(C) \Big| 
\lesssim 
\min
\Big\{
r_{\gamma_n}(C), 
\sqrt{r_{\gamma_n}(C)}
\Big\}.
\]
The key reason for our improvement is the bound~\Cref{prop:upper-bound-via-typical-radius}, which did not appear in~\cite{Cha14}, even in the Gaussian setting. Our estimates also apply for a wider class of measures. In the Gaussian setting, applying Markov's inequality to the $p$th power and optimizing over the family of $L^p$ estimates~\eqref{ineq:Lp-estimates}, we obtain the following exponential deviation inequality: 
\begin{equation}
\label{ineq:deviation-bound-in-Gauss-space}
\P\bigg\{\Big|\|\Proj{C}(G)\|_2 - r_{\gamma_n}(C)\Big| > t\,\nu(C) \bigg\}
\leq c_1 \exp\Big\{-c_2 \min\{t^2 \nu^2(C), t^4\}\Big\}
\end{equation}
for every $t > 0$ and with $\nu(C) = \min\{r_{\gamma_n}(C), \sqrt{r_{\gamma_n}(C)}\}$. The constants $c_1, c_2 > 0$ are universal. The inequality~\eqref{ineq:deviation-bound-in-Gauss-space} improves the main result (Theorem 1.1) in~\cite{Cha14}. There, the quantity $\nu(C)$ is replaced by the larger $\sqrt{r_{\gamma_n}(C)}$. A further improvement can be derived from the Borell-TIS inequality~\cite{SudTsi74,Bor75}. First, note that the nonexpansiveness of projections yields
\[
\P\bigg\{\Big|\|\Proj{C}(G)\|_2 - \E \|\Proj{C}(G)\|_2\Big| > t \bigg\} \leq 2 \e^{-t^2/2},  
\]
for all $t > 0$. Combined with~\Cref{thm:distance-from-typical-radius}, 
we obtain the inequality
\[
\P\bigg\{\Big|\|\Proj{C}(G)\|_2 - r_{\gamma_n}(C)\Big| > t\, \nu(C) \bigg\} \leq 2 \e^{-c_3 \, t^2\nu^2(C)},
\]
for any $t \geq c_4$; here $c_3, c_4 > 0$ are universal constants.
\end{remark}

We emphasize that the bounds developed above, as well as those from Chatterjee~\cite{Cha14}, are generally \emph{not} reversible. Recall that for a bounded set $T \subset \R^n$, the Euclidean outradius of $T$ is defined as
\[
\rad_2(T) = \sup_{t \in T} \, \|t\|_2.
\]

\begin{remark}[Typical radius does not fully characterize mean projection]
The ``profile'' of 
the typical radius can be compared to the norm of the projection itself along dilates of a given convex set $C \in \AllConvSets{n}$. Throughout, we assume that $(C, \mu)$ is such that $r_\mu(C)$ is well-defined; see~\Cref{prop:upper-bound-via-typical-radius}. 

First, consider the limit $\lambda \to \infty$, and the dilates $\lambda C$. 
Put $K = \cl \cone C$. By a direct calculation,
\[
\lim_{\lambda \to \infty} 
r_\mu(\lambda C) = \E h_{K \cap B^n_2}(\xi) = 
 \|\Proj{K}\|_{L^1(\mu)}  = \lim_{\lambda \to \infty} \|\Proj{\lambda C}\|_{L^1(\mu)},
\]
by \Cref{prop:qualitative-behavior-of-metric-projections}\ref{item:cone-upper-bound-and-limit}. 
Thus, the typical radius is of the correct order for large dilates. 
 
We now consider the limit as $\lambda \to 0^+$. Then, assuming that $(C, \mu)$ is not degenerate\footnote{Here, specifically, we mean that $\E_\mu h_{C \cap s B^n_2}(\xi) < \E _\mu h_C(\xi)$ for any $s \in (0, \rad_2(C))$; this holds in many cases, including Gauss space and the uniform measure on the sphere.}, we have
\[
\lim_{\lambda \to 0^+} \frac{r_\mu(\lambda C)}{\lambda} = \rad_2(C) 
\quad \mbox{while} \quad 
\lim_{\lambda \to 0^+} \frac{\|\Proj{\lambda C}\|_{L^1(\mu)}}{\lambda} = \|\SuppVec{C}\|_{L^1(\mu)}. 
\]
In many situations these quantities \emph{are not} of the same order.
Generally,
\[
\E_{\xi \sim \mu} h_C(\xi/\|\xi\|_2) \leq \|\SuppVec{C}\|_{L^1(\mu)} \leq \rad_2(C).
\]
For instance, if $\mu$ is uniform on $\{-1, 1\}^n$, these bounds imply 
\[
\frac{1}{\sqrt{n}} \rad_2(C) \lesssim \|\SuppVec{C}\|_{L^1(\mu)} \leq \rad_2(C).
\]
These inequalities are sharp in particular examples.
Indeed, fix $r > 0$. The upper bound is attained with $C_r = rB^n_2$, while the lower bound is attained (up to constants) with $C'_r = r \conv(\{e_1\} \cup \tfrac{1}{\sqrt{n}}B^n_2)$. By computation:
\[
\|\SuppVec{C_r}\|_{L^1(\mu)} = \rad_2(C_r) = \rad_2(C'_r) = r,\quad \mbox{while} \quad
\|\SuppVec{C'_r}\|_{L^1(\mu)} \asymp\frac{r}{\sqrt{n}}. 
\]
Generally, for ``pointy'' sets (such as $C'_r$), the typical radius $r_\mu$ can be much larger than the typical size of the metric projection. 
\end{remark}

\subsection{Organization}
We collect preliminaries on projections onto convex sets and
standard estimates and tail bounds in Gauss space and for spherical
measure in~\Cref{sec:preliminaries}. The proofs of the upper bounds for the extremal problems are presented in~\Cref{sec:upper-bounds}. The lower bounds are presented in~\Cref{sec:lower-bounds}. 
Proofs of the results in~\Cref{sec:bounds-on-the-expected-norm} are presented in~\Cref{sec:deferred-proofs-of-upper-bounds-on-expected-norm}.
Additional remarks on our extremal results are made in~\Cref{sec:additional-remarks}.

\subsection*{Acknowledgments} 
We thank Gil Kur and Grigoris Paouris for helpful discussions. 

\section{Preliminaries}
\label{sec:preliminaries}

In this section, we collect some basic properties of projections and basic facts about the uniform measure on the sphere and the standard Gaussian measure that will be used in the sequel.

\subsection{Properties of metric projections onto closed convex sets.}

We collect basic properties of projections onto closed, convex sets.

The next result is standard; see~\cite[Theorem 3.16]{BauCom17}.

\begin{lemma}[Variational characterization of projections]
\label{lem:projection-variational}
Let $C\subset\R^n$ be nonempty, closed, and convex. Then, for $x
\in \R^n$, 
$\Proj{C}(x) = p$ if and only if 
\begin{equation*}
\langle x- p,z-p\rangle \leq 0
\qquad\text{for every $z \in C$}.\end{equation*}
Additionally $\Proj{C}(x)$ uniquely maximizes, over $z \in C$, the mapping
\[
z\mapsto 2\langle x,z\rangle-\|z\|_2^2.
\]
\end{lemma}

\Cref{lem:projection-variational} yields the following monotonicity property along rays. 
\begin{lemma}[Monotonicity of projections]
\label{lem:pointwise-monotonicity-proj}
Let $C\in\AllConvSets{n}$ and $\theta\in\bfS^{n-1}$. Then the maps
\[
r \mapsto \|\Proj{C}(r\theta)\|_2,
\qquad
 r\mapsto \frac{\|\Proj{C}(r\theta)\|_2}{r},
\]
are nondecreasing and nonincreasing, respectively, on $(0,\infty)$.
\end{lemma}

\begin{proof}
Fix $x \in \R^n, \lambda \in (0, 1)$. By~\Cref{lem:projection-variational}, 
    \[
    \langle \lambda x - \Proj{C}(\lambda x), 
    \Proj{C}(x) - \Proj{C}(\lambda x)\rangle \leq 0,\quad 
     \langle x - \Proj{C}(x), 
    \Proj{C}(\lambda x) - \Proj{C}(x)\rangle \leq 0.
    \]
    Rearranging the second inequality and then applying the first inequality,
    \begin{multline*}
    \|\Proj{C}(x)\|_2^2  - \langle \Proj{C}(x), \Proj{C}(\lambda x)\rangle
    \\\leq \langle x, \Proj{C}(x) - \Proj{C}(\lambda x)\rangle 
    \leq \frac{1}{\lambda} \Big(\langle \Proj{C}(x), \Proj{C}(\lambda
x)\rangle -\|\Proj{C}(\lambda x)\|_2^2\Big).
    \end{multline*}
    Multiplying by $\lambda > 0$ and rearranging: 
    \begin{multline*}
    \|\Proj{C}(\lambda x)\|_2^2 +\lambda \|\Proj{C}(x)\|_2^2
    \\ \leq
(1+\lambda)
    \langle \Proj{C}(x), \Proj{C}(\lambda x)\rangle
    \leq (1+\lambda) \|\Proj{C}(x)\|_2 \|\Proj{C}(\lambda x)\|_2.
    \end{multline*}
    Solving the inequality $a^2 + \lambda b^2 \leq (1+\lambda) ab$ for 
    $a = \|\Proj{C}(\lambda x)\|_2, b = \|\Proj{C}(x)\|_2$, 
    \[
    \lambda \|\Proj{C}(x)\|_2 \leq \|\Proj{C}(\lambda x)\|_2 \leq
\|\Proj{C}(x)\|_2.    
\]
     Taking, \eg $x = r\theta$, we obtain the claim.
\end{proof}

The next observation is obvious from the definition of metric projections. It implies that on the sphere, we may reduce to the case that $C \subset B^n_2$.

\begin{lemma}
\label{lem:truncation-projection}
Let $C\in\AllConvSets{n}$ and put $K=C\cap B_2^n$. Then
\[
\Proj{C}(\theta)=\Proj{K}(\theta)\qquad \text{for every }\theta\in\bfS^{n-1}.
\]
\end{lemma}


Next, we give the limit of the projection along rays.
Recall~\Cref{def:minimal-norm-supp-vec}.

\begin{lemma}
\label{lem:ray-limit-of-projections}
Suppose $C \subset \R^n$ is a nonempty, compact convex set. Then 
\[
\lim_{\lambda \to \infty} \Pi_C(\lambda x) = \SuppVec{C}(x), \quad \mbox{for any}~x
\in \R^n.
\]    
\end{lemma}
\begin{proof}
Consider the nonempty, compact, convex set 
\[
F_x= \Big\{\, z\in C \mid \langle z,x\rangle=h_C(x) \, \Big\}.
\] 
The point $\SuppVec{C}(x)$ is the (unique) minimal norm element of $F_x$.
By~\Cref{lem:projection-variational}, 
\[
0\leq
2\Big(h_C(\lambda x)-\langle \lambda x, \Proj{C}(\lambda
x)\rangle\Big) 
\leq \|\SuppVec{C}(x)\|_2^2-\|\Proj{C}(\lambda x)\|_2^2
\leq \|\SuppVec{C}(x)\|_2^2,
\]
for every $\lambda > 0$. Consequently, by $1$-positive homogeneity of $h_C$, we
see
that
\[
\lim_{\lambda \to \infty} \langle x, \Proj{C}(\lambda x)\rangle = h_C(x).
\]
Moreover, we have
$\|\Proj{C}(\lambda x) \|_2\leq\|\SuppVec{C}(x)\|_2$ for every $\lambda > 0$. Hence,
any accumulation point 
$p$ of $\{\Proj{C}(\lambda x)\}_{\lambda > 0}$
belongs to $F_x$
and satisfies
$\|p\|_2\leq\|\SuppVec{C}(x)\|_2$, whence $p=\SuppVec{C}(x)$ by the minimality and
uniqueness of $\SuppVec{C}(x)$. The conclusion now follows from the compactness of $C$.
\end{proof}

We use the following variational interpretation of the norm of the projection. Recall that a function $f \colon \R^n \to \R$ is \emph{$1$-strongly concave} if the map
\[
x \mapsto f(x) + \frac{\|x\|_2^2}{2}
\]
is concave.

\begin{lemma}
\label{lem:radius-interpretation-of-projection}
    For any $C \in \AllConvSets{n}$ and any $x \in \R^n$, 
    the map 
    \[
    r \mapsto h_{C \cap r B^n_2}(x) - \frac{r^2}{2}
    \]
    is $1$-strongly concave and uniquely maximized on $\R_+$ at $\|\Proj{C}(x)\|_2$. 
\end{lemma}
\begin{proof}
For $r \geq 0$, define the functions 
\[
\psi_x(r)=h_{C\cap rB_2^n}(x),\qquad 
\phi_x(r)=\psi_x(r)-\frac{r^2}{2},\qquad r\geq0.
\]
First, by the convexity of $C$, 
\[
C \cap (tr + (1-t)s)B^n_2 \supset t(C \cap r B^n_2) + (1-t) (C\cap sB^n_2),
\]
for any $r,s\geq0$, and 
$t\in[0,1]$. Hence, $\psi_x$ is concave and thus 
$\phi_x$ is $1$-strongly concave as 
$\phi_x(r)+r^2/2=\psi_x(r)$.
Now, set 
\[
p=\Proj{C}(x), \quad \mbox{and} \quad 
\rho=\|p\|_2.
\]
By~\Cref{lem:projection-variational}, the point $p$ uniquely maximizes
\[
z\mapsto \langle x,z\rangle-\frac{\|z\|_2^2}{2}
\quad \mbox{over}~z \in C.
\]
For every $r\geq0$,
\begin{multline*}
\phi_x(r)
=\sup_{z\in C\cap rB_2^n}\Big\{\langle x,z\rangle-\frac{r^2}{2}\Big\}
\\\leq
\sup_{z\in C\cap rB_2^n}\bigg\{\langle x,z\rangle-\frac{\|z\|_2^2}{2}\bigg\}
\leq
\langle x,p\rangle-\frac{\|p\|_2^2}{2}.
\end{multline*}
On the other hand, since $p\in C\cap \rho B_2^n$,
\[
\phi_x(\rho)\geq \langle x,p\rangle-\frac{\rho^2}{2}
=\langle x,p\rangle-\frac{\|p\|_2^2}{2}.
\]
Thus $\rho$ maximizes $\phi_x$ on $\R_+$. Since $\phi_x$ is strongly concave,
this maximizer is unique.
\end{proof}

\subsection{Properties of spherical and Gaussian measure}

We recall standard estimates for spherical and Gaussian measure.

The first result follows from Lévy's concentration for Lipschitz
functionals on the sphere (\eg~\cite[Theorem~5.1.3]{Ver26}):
\[
\|f - \E f\|_{\psi_2(\sigma_{n-1})} 
\lesssim \frac{\|f\|_{\rm Lip}}{\sqrt{n}},
\]
for any $f \colon \bfS^{n-1} \to \R$ that is Lipschitz.\footnote{Adjusting 
constants, this holds with either the geodesic or Euclidean metric on
$\bfS^{n-1}$.}

\begin{lemma}
\label{lem:spherical-concentration}
If $g:\bfS^{n-1}\to\R$ is Lipschitz, then, for every $1\leq q\leq n$,
\[
\|g-\E g\|_{L^q(\sigma_{n-1})}\lesssim \|g\|_{\rm Lip}\sqrt{\frac qn}.
\]
\end{lemma}

We require some estimates for fixed marginals of a
random direction.

\begin{lemma}
\label{lem:coordinate-moments}
For every $n \geq 1$, every $u \in \bfS^{n-1}$
and
every
$1 \leq r \leq\infty$,
\[
\|\langle \theta, u \rangle_+\|_{L^r(\sigma_{n-1})}\asymp \sqrt{\frac{\twomin{r}{n}}{n}}.
\]
\end{lemma}

\begin{proof}
For $r = \infty$, it clearly holds that $\|\langle \theta, u\rangle_+\|_{L^\infty(\sigma_{n-1})} = 1$, as required. For $1 \leq r < \infty$, integrating a Gaussian random vector $G \sim \gamma_n$ in polar coordinates
and using rotational invariance yields
\[
\|\langle \theta, u \rangle_+\|_{L^r} = 
\frac{\|(G_1)_+\|_{L^r}}{\|\chi_n\|_{L^r}} 
\asymp \frac{\|G_1\|_{L^r}}{\|\chi_n\|_{L^r}} \asymp \sqrt\frac{r}{\twomax{r}{n}}\asymp
\sqrt\frac{\twomin{r}{n}}{n}.\qedhere
\]
\end{proof}

We also require some tail bounds on such marginals. 
Note that the constants given below are not optimized. 

\begin{lemma}
\label{lem:spherical-estimates}
Fix $n \geq 3$. 
Let $T = \langle \theta, u\rangle$ where $\theta \sim \sigma_{n-1}$
and $u \in \bfS^{n-1}$ is fixed. Then: 
\begin{enumerate}[label=(\roman*)]
\item 
\label{ineq:small-ball-inequality}
For $t \in (0, 1)$ it holds that 
\[
\P\{|T| \leq t\} \asymp \min\big\{1, t\sqrt{n}\big\}.
\]
\item 
\label{ineq:lower-bound-on-lower-tail}
For $t \in (0, 1/12)$ it holds that 
\[
\frac{1}{12} \e^{-12nt^2} \leq  \P\{T \geq t\} \leq 12 \e^{-nt^2/12}.
\]
\end{enumerate}
\end{lemma}
\begin{proof}
From~\cite[eqn.~(5.3)]{AubSza17}, we have 
\[
\P\{|T| \leq t\} \asymp \sqrt{n} \int_0^{\sin^{-1}(t)} 
\cos^{n-2}(\theta)\,\ud \theta. 
\]
If $t \leq \sqrt{1/n}$, then using $1 \geq \cos(u) \geq 1 - \tfrac{u^2}{2}$
for all
$u
\geq 0$ and $\sin^{-1}(t) \asymp t$, 
\[
t\sqrt{n} \gtrsim\P\{|T| \leq t\} \gtrsim \sqrt{n} \int_0^{t} \Big(1 -
\frac{\theta^2}{2}\Big)^{n-2} \, \ud \theta \gtrsim t \sqrt{n}. 
\]
Above, we used $(1-\tfrac{\theta^2}{2})^{n-2} \geq
(1-\tfrac{1}{2n})^{n-2}\geq \e^{-1/2}$. Additionally, if $t \geq \sqrt{1/n}$,
then applying the above inequality at $t = 1/\sqrt{n}$, we obtain 
\[
1 \geq \P\{|T| \leq t\} \geq \P\{|T| \leq \sqrt{1/n}\} \asymp 1.
\] 
Combining the cases, we obtain claim~\ref{ineq:small-ball-inequality}.

For the second set of inequalities, we use known bounds
from~\cite{BriGriKanKleLovSim01}. We set  
\[
F_n(t) = \P\{T \geq t\}.
\]
For $n \geq 3$ and $t \in [\sqrt{2/n}, 1]$,
by~\cite[Lemma~2.1(ii)]{BriGriKanKleLovSim01}
\[
F_n(t) \leq \frac{1}{2t \sqrt{n}} (1-t^2)^{(n-1)/2} \leq \frac{1}{2\sqrt{2}}
\e^{-nt^2/4}
\leq C \e^{-nt^2/C},
\]
for $C \geq 4$. Finally, by~\cite[Lemma~2.1(ii)]{BriGriKanKleLovSim01}, we have
for $t \in [\sqrt{2/n}, 1/12]$ that 
\[
F_n(t) \geq \frac{1}{6 t\sqrt{n}} 
(1-t^2)^{(n-1)/2}
\geq \frac{1}{6 t\sqrt{n}} \e^{-nt^2}
\geq \e^{-3nt^2}.
\]
Above, we used $1-x\geq \e^{-2x}$ for $x \in [0, 1/2]$ and $1 \geq
6u\e^{-2u^2}$ for $u \geq \sqrt{2}$. 
It remains to prove the lower bound when $t\leq \sqrt{2/n}$.
Write the density of $T$ as
\[
f_n(s)=c_n(1-s^2)^{(n-3)/2},\qquad 
c_n=\frac{\Gamma(n/2)}{\sqrt{\pi}\Gamma((n-1)/2)}\leq \sqrt n.
\]
If $0\leq t\leq 1/(3\sqrt n)$, then by symmetry
\[
F_n(t)=\frac12-\int_0^t f_n(s)\,ds
\geq \frac12-t\sqrt n
\geq \frac16
\geq \frac1{12}e^{-12nt^2}.
\]
If $1/(3\sqrt n)\leq t\leq \sqrt{2/n}$, then by monotonicity and
\cite[Lemma~2.1(ii)]{BriGriKanKleLovSim01},
\[
F_n(t)\geq F_n(\sqrt{2/n})
\geq \frac{1}{6\sqrt2}(1-2/n)^{(n-1)/2}
\geq \frac{1}{18\sqrt2}
\geq \frac1{12}e^{-4/3}
\geq \frac1{12}e^{-12nt^2}.
\]
Combining all the cases, we obtain
bounds~\ref{ineq:lower-bound-on-lower-tail}.
\end{proof}

One basic consequence of~\Cref{lem:spherical-estimates} is the
following estimate. 
\begin{lemma}
\label{lem:truncated-lp-estimate}
Fix $n \geq 3$. Let $T = \langle \theta, u\rangle$ where 
$\theta \sim \sigma_{n-1}$ and $u \in \bfS^{n-1}$ is fixed. Then for $t
\in (0, 1)$, and $1 \leq p < \infty$, it holds that
\[
\|(T - t)_+\|_{L^p}
\lesssim \frac{p}{n} \frac{1}{t} \e^{-nt^2/(12 p)}.
\]
\end{lemma}
\begin{proof}
We have 
\[
\E (T-t)_+^p = p \int_0^{1-t} s^{p-1} \,\P\{ T > t + s\}
\, \ud s. 
\]  
From the proof of~\Cref{lem:spherical-estimates}, it holds that 
\[
\P\{T > u\} \leq 12 \exp\{-nu^2/12\}, \quad \mbox{for}~u\in(0,1).
\]
Combining the two preceding displays, we obtain:
\begin{align*}
\E (T-t)_+^p
&\leq 12 p \int_0^{1-t} s^{p-1} \, \e^{-n(t+s)^2/12} \, \ud s \\ 
&\leq 12 p \e^{-nt^2/12} \int_0^{\infty} s^{p-1} \e^{-nts/6} \, \ud s = 
12 \Big(\frac{6}{t n}\Big)^p \Gamma(p+1) \e^{-nt^2/12}.
\end{align*}
Above, we used $(t+s)^2 \geq t^2 +2ts$. Taking $p$th roots on both sides, 
\[
\|(T-t)_+\|_{L^p} \leq \frac{72}{tn} (\Gamma(p+1))^{1/p} \e^{-nt^2/(12p)}
\leq \frac{144}{t} \frac{p}{n} \e^{-nt^2/(12p)},
\]
where we used $\Gamma(p+1)
\leq
(2p)^p$ for $p \geq 1$.
\end{proof}

We also need a basic estimate for the $L^p$ norm of a $\chi_n$ variate, which
we will later use to control the norm $\|G\|_2$ when $G \sim \gamma_n$.

\begin{lemma}
\label{lem:chi-n-concentration}
For $n\geq1$ and $1\leq p<\infty$, it holds that 
\[
\|\chi_n\|_{L^p}\asymp \sqrt{\twomax{p}{n}}.
\]
\end{lemma}

\begin{proof}
Since $\|G\|_2 \sim \chi_n$ when $G \sim \gamma_n$, we have 
\[
\|\chi_n\|_{L^p} \geq \max\{\|G_1\|_{L^p(\gamma_n)}, \|\chi_n\|_{L^1}\}
\gtrsim \max\{\sqrt{p}, \sqrt{n}\} = \sqrt{\twomax{p}{n}}.
\]
In the other direction, the triangle inequality yields:
\[
\|\chi_n\|_{L^p} \lesssim  
\E \chi_n + \sqrt{p}\|\chi_n - \E \chi_n\|_{\psi_2} 
\lesssim \sqrt{n} + \sqrt{p} \asymp \sqrt{\twomax{p}{n}}.
\]
Here, we used the estimate (\eg \cite[Theorem 3.1.1]{Ver26}), $\|\chi_n -
\E
\chi_n\|_{\psi_2} \lesssim
1$.
\end{proof}

\section{Proofs of upper bounds}
\label{sec:upper-bounds}

In this section, we give the proofs of the upper bounds for \Cref{thm:optimal-reverse-holder-sphere,thm:optimal-reverse-holder-Gauss}.  
We first state and prove an intermediate result: a lower bound on the $L^p$ norm of projections, which will allow us to control the ratio of $L^q$ to $L^p$ norms from above. 

\subsection{Lower bounds on the $L^p$ norm of projections}
In this section, we establish a lower bound on the 
$L^p(\sigma_{n-1})$
norm
of the projection onto $C \in \AllConvSets{n}$, in terms of its radius. 
\begin{theorem}
\label{thm:radius-p-moment}
For $n \geq 3$, 
for every $C\in\AllConvSets{n}$, $C \neq \{0\}$, 
with $C \subset B^n_2$, it holds that
\[
\|\Proj{C}\|_{L^p(\sigma_{n-1})}
\gtrsim 
\rad_2(C) \sqrt{\frac{p}{n}
\min\left\{\log\frac{\e n}{p},\,\log\frac{\e}{\rad_2(C)}\right\}},
\]
for any $1 \leq p \leq n$. 
\end{theorem}

\begin{remark}[Sharpness of~\Cref{thm:radius-p-moment}]
\Cref{thm:radius-p-moment}
is sharp. Consider
\[
C(a, \rho) = \rho \, \conv(aB^n_2, e_1), \quad \mbox{where}~\rho \in (0, 1),~ a \in (0,\tfrac{1}{12}).
\]
From~\Cref{lem:projection-norm-for-C}, 
\begin{align*}
\|\Proj{C(a, \rho)}\|_{L^p(\sigma_{n-1})} 
&\lesssim a\rho 
+ \|\min\{\rho, 
(\theta_1 - a)_+\}\|_{L^p(\sigma_{n-1})} \\ 
&\lesssim a\rho + \min\Big\{\rho, \frac{p}{n a}\Big\}\, 
\e^{-na^2/(12p)},
\end{align*}
by standard spherical estimates (\eg
\Cref{lem:spherical-estimates}
and 
\Cref{lem:truncated-lp-estimate}). 
Put 
\[
a^\star_{\rho, n,p} = \sqrt{12 \, \frac{p}{n} \min\Big\{\log \frac{\e n}{p}, \log \frac{\e}{\rho}\Big\}}. 
\]
Note $a^\star_{\rho, n,p} \leq \tfrac{1}{12}$ if $1 \leq p \leq c n$ for a sufficiently small $c > 0$. Hence,
\[
\|\Proj{C(a^\star_{\rho, n,p}, \rho)}\|_{L^p(\sigma_{n-1})}  
\lesssim  \rho \, a^\star_{n,p, \rho} + 
\rho\frac{p}{n a^\star_{n,p,\rho}} 
\lesssim \rho \, a^\star_{n,p, \rho}.
\]
Thus, for any $1 \leq p \leq cn$, since $\rad_2(C(a^\star_{\rho, n,p}, \rho)) = \rho$, the estimate in~\Cref{thm:radius-p-moment} is 
unimprovable, apart from constant factors. If $p \geq cn$, then the lower bound~\Cref{thm:radius-p-moment} is 
of order $\rad_2(C)$ and thus unimprovable. 
\end{remark}

The papers of Gromov~\cite{Gro03}, 
Giannopoulos-Milman-Tsolomitis~\cite{GiaMilTso05}, and
Vershynin~\cite{Ver06} established relationships between the existence of 
bounded sections of centrally symmetric convex bodies and their typical (\ie random) sections, using the so-called ``isoperimetry of waists.'' This is connected to projections, but yields a weaker estimate.

\begin{remark}[Interpretation in terms of inequalities for ``waists'']
The waist inequality implies that if a centrally symmetric convex body $K \subset \R^n$ has a $k$-dimensional projection containing the unit ball, then its $\varepsilon$-neighborhood intersects a large portion of the sphere. This is related to projections via:
\[
(K + \varepsilon B^n_2) \cap \bfS^{n-1} \subset \{\theta \in \bfS^{n-1} : \|\Proj{K}(\theta)\|_2 \geq 1 -\varepsilon \},
\]
which holds for any $\varepsilon \in (0, 1)$. If $K = \rho C$ for a centrally symmetric convex body $C \in \AllConvSets{n}$ with $C \subset B^n_2$, by~\cite[Proposition 3.1]{Ver06}, we obtain
\[
\|\Proj{C}\|_{L^p(\sigma_{n-1})} 
\geq \rad_2(C)\, 
(1-\varepsilon) \, \Big[\P_{\theta \sim \sigma_{n-1}}\{\theta_1 > \sqrt{1-\varepsilon^2}\}\Big]^{1/p}, 
\] 
for any $\varepsilon \in (0, 1)$ and $1 \leq p \leq n$. Optimizing the right-hand side yields 
\[
\|\Proj{C}\|_{L^p(\sigma_{n-1})} \gtrsim \rad_2(C) \, \frac{p}{n},
\]
which is weaker than \Cref{thm:radius-p-moment}. 
Our improvement comes from constructing a dichotomy for projections along certain meridians; see the discussion before the proof of \Cref{thm:radius-p-moment} below.
\end{remark}

An initial estimate bounds the projection by the support function. 

\begin{lemma}
\label{lem:support-lower-bound}
For every $C\in\AllConvSets{n}$ with $C \subset B^n_2$, it holds that: 
\begin{enumerate}[label=(\roman*)] 
\item 
\label{item:lower-bound-by-support-function}
for every $\theta \in \bfS^{n-1}$,    
it holds that $\|\Proj{C}(\theta)\|_2 \geq \tfrac{1}{2}
h_C(\theta)$; and, 
\item 
\label{item:lower-bound-from-support-function}
for every $1 \leq p \leq \infty$, 
\[
\|\Proj{C}\|_{L^p(\sigma_{n-1})}
\gtrsim
\rad_2(C) \sqrt{\frac{\twomin{p}{n}}{n}}
\geq \|\Proj{C}\|_{L^\infty(\sigma_{n-1})}\sqrt{\frac{\twomin{p}{n}}{n}}.
\]
\end{enumerate}
\end{lemma}

\begin{proof}
By \Cref{lem:projection-variational}, for
all $z\in C$,
\[
\langle \theta,z\rangle
\leq
\langle\theta,\Proj{C}(\theta)\rangle+
\langle \Proj{C}(\theta),z\rangle-\|\Proj{C}(\theta)\|_2^2
\leq 2\|\Proj{C}(\theta)\|_2,
\]
because $C\subset B_2^n$.  Thus
$h_C(\theta)\leq2\|\Proj{C}(\theta)\|_2$, which
establishes claim~\ref{item:lower-bound-by-support-function}. 
We can pick a direction $u \in
\bfS^{n-1}$
such
that $r u \in C$ where $r = \rad_2(C)$. Since
$0
\in
C$, we have
\[
\|\Proj{C}(\theta)\|_2\geq \frac{1}{2} \, h_C(\theta) 
\geq 
\frac{1}{2} \max\{r \langle \theta, u\rangle, 0 \} 
= \frac{r}{2} \langle\theta, u\rangle_+.
\]
Now, claim~\ref{item:lower-bound-from-support-function} follows
from~\Cref{lem:coordinate-moments}.
\end{proof}

The next lemma shows that convexity forces the projection to be large along the
meridian 
\[
y_s(u, v) = su + \sqrt{1-s^2}\,v 
\quad \mbox{for}~u \in \bfS^{n-1}, v \in \bfS^{n-1} \cap u^\perp, |s| \leq 1.
\]
For $u\in\bfS^{n-1}$, define the radial function of $C$ by
\[
\rho_C(u)=\sup\{r\geq0:ru\in C\}.
\]
\begin{lemma}
\label{lem:lower-on-meridian}
For $n \geq 2$, fix $u, v \in \bfS^{n-1}$ with $v \in u^\perp$. Fix $t
\in [-1,1]$. Then for any $C \in \AllConvSets{n}$ with $C \subset B^n_2$, the following
hold:
\begin{enumerate}[label=(\roman*)]
\item 
\label{item:lower-for-t-near-zero}
if $|t| \leq \tfrac{1}{2} \tfrac{h_C(v)}{\rad_2(C)}$, then
$\|\Proj{C}(y_t(u, v))\|_2 \geq \tfrac{1}{8} h_C(v)$; and, 
\item 
\label{item:lower-away-from-zero}
for any $t \in (0, 1]$, 
\[
\|\Proj{C}(y_t(u,v))\|_2 \geq \frac{(\rho_C(u) t - h_C(v))_+}{\rho_C(u) + t}.
\]   
\end{enumerate}
\end{lemma}
\begin{proof}
Let $w \in C$ be such that $h_C(v) = \langle v, w\rangle$. Then, 
\[
h_C(y_t(u,v)) \geq t \langle u, w\rangle + \sqrt{1-t^2} h_C(v)
\geq \sqrt{1-t^2} h_C(v) - |t| \rad_2(C).
\]    
Now applying the assumption $|t| \leq \tfrac{1}{2} \tfrac{h_C(v)}{\rad_2(C)}$,
which also implies $|t| \leq 1/2$, 
\[
h_C(y_t(u,v)) \geq \Big(\sqrt{\frac{3}{4}} - \half \Big) h_C(v) \geq
\frac{1}{3} h_C(v).
\]
Inequality~\ref{item:lower-for-t-near-zero} then follows 
by
applying~
\Cref{lem:support-lower-bound}\ref{item:lower-bound-by-support-function}.
Now, applying the variational characterization
in~\Cref{lem:projection-variational}, we obtain with $\rho = \rho_C(u)$,
\begin{align*}
0 &\geq \langle y_t(u,v) - \Proj{C}(y_t(u,v)), \rho u -
\Proj{C}(y_t(u,v))\rangle \\
&\geq \rho t  
-\sqrt{1-t^2} \langle v, \Proj{C}(y_t(u,v))\rangle 
-(\rho + t)\langle u, \Proj{C}(y_t(u,v))\rangle\\
&\geq \rho t - h_C(v) - (\rho + t) \|\Proj{C}(y_t(u,v))\|_2. 
\end{align*}
Rearranging, we obtain for $t \geq 0$,  
\[
\|\Proj{C}(y_t(u,v))\|_2 \geq \frac{\rho t - h_C(v)}{\rho + t}.
\]
Since the left-hand side is nonnegative, this implies
inequality~\ref{item:lower-away-from-zero}.
\end{proof}

We now establish~\Cref{thm:radius-p-moment}. The idea of the proof is to construct a dichotomy showing that the projection must be large along the meridian on the sphere that corresponds to the direction achieving the outradius of $C$. Let $u \in \bfS^{n-1}$ achieve the outradius; that is, $\rho_C(u) = \rad_2(C)$, where $\rho_C$ denotes the radial function of $C$. We decompose $\theta \in \bfS^{n-1}$ as
\[
\theta = tu + \sqrt{1-t^2} \, v \quad \mbox{for some}~v \in \bfS^{n-1} \cap u^\perp,~ t \in [-1,1].
\]
We show that either the support function of $v$ is large, and hence there is a ``tube'' around the $u$-equator (\ie an interval of $t$ near $0$) on which the projection is large, or the projection can be a bit smaller but this is compensated for by occurring on a sufficiently large spherical cap. 
\begin{proof}[Proof of~\Cref{thm:radius-p-moment}]
We may assume that $\rad_2(C) > 0$; otherwise the claim is trivial. Let $\rho =
\rad_2(C)$ and fix $u \in \bfS^{n-1}$ such that $\rho_C(u) = \rad_2(C) = \rho$.
Denote by 
\[
L \equiv L(n,p, \rho) = \sqrt{\frac{p}{n} \min\Big\{
\log\frac{\e n}{p}, \log
\frac{\e}{\rho}\Big\}}.
\]
Define for $v \in \bfS^{n-1} \cap u^\perp$, 
\[
m_{C,p}(v) = \Big(\E_{T} \|\Proj{C}(Tu + \sqrt{1-T^2} v)\|_2^p\Big)^{1/p}
\]
where $T$ has the distribution of $\langle \theta, u \rangle$ when $\theta
\sim \sigma_{n-1}$. 
Note that 
\[
\|\Pi_C\|_{L^p(\sigma_{n-1})} = \|m_{C,p}\|_{L^p(\mathrm{Unif}(\bfS^{n-1} \cap u^\perp))}.
\]
Consequently, by the above display, it suffices to show that
\begin{equation}
\label{ineq:meridian-lower-bound}
m_{C,p}(v) \gtrsim \rho L, \quad \mbox{for all}~v \in \bfS^{n-1} \cap u^\perp.
\end{equation}
To that end, we fix $v \in u^\perp$ and denote $T = \langle \theta,
u\rangle$ where $\theta \sim \sigma_{n-1}$. We also denote 
\[
\alpha(v) = \frac{h_C(v)}{\rad_2(C)} = \frac{h_C(v)}{\rho}.
\]

\paragraph{Case 1: $\alpha(v) \geq \tfrac{1}{24} L$}
By~\Cref{lem:lower-on-meridian}\ref{item:lower-for-t-near-zero}, we obtain 
\[
m_{C, p}(v)\geq \frac{1}{8} \rho \alpha(v) \P\{|T| \leq \tfrac{1}{2}
\alpha(v)\}^{1/p} \geq \frac{1}{192} \rho L
\P\{|T| \leq \tfrac{1}{48} L\}^{1/p}. 
\]
Note that $L \leq 1$ by definition; thus
by~\Cref{lem:spherical-estimates}\ref{ineq:small-ball-inequality}, \[
m_{C,p}(v) \gtrsim \rho L \min\{1, (L \sqrt{n})^{1/p}\} = \rho L,
\]
which establishes inequality~\eqref{ineq:meridian-lower-bound} in this case.
Above, we used 
$L\sqrt{n} \geq \sqrt{p} \geq 1$. 

\paragraph{Case 2: $\alpha(v) \leq \tfrac{1}{24} L$} 
Set $t = L/12$. By~\Cref{lem:lower-on-meridian}\ref{item:lower-away-from-zero}, 
\begin{equation}
\label{ineq:lower-bound-on-meridian-two}
m_{C,p}(v) \geq \rho \, 
\frac{(t - \alpha(v))_+}{\rho + t} \P\{T \geq \tfrac{1}{12}L\}^{1/p}
\gtrsim \frac{\rho L}{\rho + L} 
\P\{T \geq \tfrac{1}{12} L\}^{1/p}.
\end{equation}
Then, since $L \leq 1$, by
\Cref{lem:spherical-estimates}\ref{ineq:lower-bound-on-lower-tail},
\begin{multline}
\label{ineq:lower-bound-on-meridian-three}
\P\{T \geq \tfrac{1}{12}L\}^{1/p} \\\gtrsim \exp\Big(-\frac{1}{12}
\min\Big\{\log \frac{\e n}{p}, \log \frac{\e}{\rho}\Big\}\Big)
\asymp  
\Big(\frac{p}{n} + \rho \Big)^{1/12} 
\gtrsim \rho + L.
\end{multline}
Above, the final inequality arises from the numerical inequality
\begin{equation}
\label{ineq:numerical-inequality}
x + \sqrt{y
\min\{\log(\e/y),
\log(\e/x)\}} \lesssim (x+y)^{1/12}, 
\end{equation}
which holds for all $x, y \in (0, 1]$. Indeed, observe that $x \leq x^{1/12} \leq (x+y)^{1/12}$. Additionally,
\[
\sqrt{y
\min\{\log(\e/y),
\log(\e/x)\}} 
\leq y^{1/12} \cdot \sup_{z \in (0, 1]} \sqrt{z^{5/6} \log(\e/z)} \lesssim y^{1/12},
\]
since $z \mapsto z^{5/6} \log (\e/z)$ is bounded on $(0, 1]$. 
Combining both bounds yields~\eqref{ineq:numerical-inequality}.
We applied this inequality with
$x=
\rho,
y
=
p/n$.
Combining inequalities~\eqref{ineq:lower-bound-on-meridian-two}
and~\eqref{ineq:lower-bound-on-meridian-three}
yields~\eqref{ineq:meridian-lower-bound}, as required.
\end{proof}

\subsection{Upper bound for the spherical case}

We are now in a position to prove the upper bound when $\mu = \sigma_{n-1}$.

\begin{proof}[Proof of~\Cref{thm:optimal-reverse-holder-sphere} (upper bound)]
Fix a convex set $C\in\AllConvSets{n}$, with $C\neq\{0\}$. By \Cref{lem:truncation-projection}, without loss of
generality, we have
$C
\subset
B^n_2$.  
Throughout, we denote the exponent
\[
\delta=1-\frac{p}{q} = 2\tau. 
\]

By Hölder interpolation
and the lower bound from
\Cref{lem:support-lower-bound}, we have 
\[
\|\Proj{C}\|_{L^q(\sigma_{n-1})} \leq 
\|\Proj{C}\|_{L^p(\sigma_{n-1})}^{1-\delta}
\|\Proj{C}\|_{L^\infty(\sigma_{n-1})}^\delta
\lesssim \Big(\frac{n}{\twomin{p}{n}}\Big)^{\delta/2}
\|\Proj{C}\|_{L^p(\sigma_{n-1})}.
\] 
If $q\geq n$---including $q=\infty$---the display above implies the desired
upper bound.

Now suppose $q < n$ and that $n \leq N$ for some 
sufficiently large positive integer $N$. 
We have 
\[
1 \leq \log(\e n/\twomin{q}{n})^{\delta/2} \leq \log(\e N).
\] 
Therefore, for a sufficiently large universal
constant, 
\[
\|\Proj{C}\|_{L^q(\sigma_{n-1})} \lesssim 
\Big(\frac{n}{(\twomin{p}{n}) \log(\e \tfrac{n}{\twomin{q}{n}})}\Big)^{\delta/2}
\|\Proj{C}\|_{L^p(\sigma_{n-1})}, 
\]
for $n \leq N$ and $1 \leq p < q < n$, 
which again furnishes the desired upper bound. 

Finally, suppose $1 \leq p < q < n$ and $n > N$.  
Then, $\twomin{p}{n}=p$ and $\twomin{q}{n}=q$.
We now split into cases, depending on the size of the radius $\rho = \rad_2(C)$. 

If $\rho\leq \sqrt{\tfrac{q}{n}}$, then
$\log(\e/\rho)\gtrsim \log(\e n/q)$, while
$\log(\e n/p)\geq \log(\e n/q)$. From \Cref{thm:radius-p-moment}, it holds
that
$\|\Proj{C}\|_{L^p(\sigma_{n-1})}
\gtrsim\rho\sqrt{p\log(\e n/q)/n}$, whence 
\[
\|\Proj{C}\|_{L^q(\sigma_{n-1})} 
\leq \rho^\delta \|\Proj{C}\|_{L^p(\sigma_{n-1})}^{1-\delta} \lesssim  
\left(\frac{n}{p\log(\e n/q)}\right)^{\delta/2} 
\, \|\Proj{C}\|_{L^p(\sigma_{n-1})},
\] 
by Hölder interpolation. 

If $\rho>\sqrt{q/n}$, then $1\leq \log(\e/\rho)\leq \log(\e n/q) \leq
\log(\e n/p)$. Hence
\Cref{thm:radius-p-moment} gives:
\begin{equation}
\label{ineq:upper-interp-log-1}
\|\Proj{C}\|_{L^p(\sigma_{n-1})}
\gtrsim \rho\sqrt{\frac{p}{n}\log \frac{\e}{\rho}}.
\end{equation} 
The same Hölder interpolation argument yields 
\begin{equation}
\label{ineq:upper-interp-log-2}
\|\Proj{C}\|_{L^q(\sigma_{n-1})} 
\lesssim     \left(\frac{n}{p\log(\e/\rho)}\right)^{\delta/2} \,
\|\Proj{C}\|_{L^p(\sigma_{n-1})}.
\end{equation}
On the other hand, since $\theta\mapsto\|\Proj{C}(\theta)\|_2$ is
$1$-Lipschitz, we may apply the Lipschitz concentration inequality on the sphere
(\Cref{lem:spherical-concentration}), which yields:
\begin{equation}
\label{ineq:upper-interp-log-3}
\|\Proj{C}\|_{L^q(\sigma_{n-1})} 
\lesssim \sqrt{\frac{q}{n}} + \|\Proj{C}\|_{L^1(\sigma_{n-1})}
\lesssim \sqrt{\frac{q}{n}}
+\|\Proj{C}\|_{L^p(\sigma_{n-1})}.
\end{equation}
Combining inequalities~\eqref{ineq:upper-interp-log-1} 
and~\eqref{ineq:upper-interp-log-3}, we obtain 
\begin{multline}
\label{ineq:upper-conc-log}
\|\Proj{C}\|_{L^q(\sigma_{n-1})} \lesssim 
\Big(1 + \frac{\sqrt{q}}{\rho \sqrt{p\log(\e/\rho)}}\Big) \,
\|\Proj{C}\|_{L^p(\sigma_{n-1})}  \\
\lesssim 
\frac{\sqrt{q}}{\rho \sqrt{p\log(\e/\rho)}} 
\, \|\Proj{C}\|_{L^p(\sigma_{n-1})}.
\end{multline}
If $\log(\e/\rho) \geq \tfrac{1}{4} \delta \log(\e n/q)$, then
by inequality~\eqref{ineq:upper-interp-log-2}, 
\[
\|\Proj{C}\|_{L^q(\sigma_{n-1})} 
\lesssim     \left(\frac{n}{p\log(\e n/q)}\right)^{\delta/2} \,
\|\Proj{C}\|_{L^p(\sigma_{n-1})},
\] 
where we used that $(c\delta)^{-\delta/2} \lesssim 1$ for $\delta
\in (0, 1]$. Finally, if $\log(\e/\rho) < \tfrac{1}{4}\delta \log(\e n/q)$,
then
by the bound~\eqref{ineq:upper-conc-log}, 
\[
\|\Proj{C}\|_{L^q(\sigma_{n-1})} 
\lesssim \sqrt{A(n,p,q,\rho)} \left(\frac{n}{p\log(\e n/q)}\right)^{\delta/2}
\,
\|\Proj{C}\|_{L^p(\sigma_{n-1})},
\]
where we have defined
\[
A(n,p, q,\rho) = 
\frac{1}{\rho^2 \log(\e/\rho)} 
\frac{q \log^\delta(\e n/q)}{p^{1-\delta} n^\delta}.
\]
It suffices to show that $A(n,p,q,\rho) \lesssim 1$, and then the claim
follows.
Note
that $(q/p)^{p/q} \lesssim 1$ and $\log(\e/\rho) \gtrsim 1$, hence 
\[
A(n, p, q, \rho) \lesssim \frac{1}{\rho^2}
\Big(\frac{q}{n}\Big)^\delta \, \log^{\delta}(\e n/q).
\]
Additionally, $\log(\e/\rho) < \tfrac{1}{4}\delta \log(\e n/q)$ implies
\[
\frac{1}{\rho^2} \leq \Big(\frac{n}{q}\Big)^{\delta/2}
\e^{\delta/2-2} \lesssim \Big(\frac{n}{q}\Big)^{\delta/2}.
\]
Combining the previous two displays,
\[
A(n,p, q, \rho) \lesssim \Big(\frac{\log(\e n/q)}{\sqrt{n/q}}\Big)^\delta 
\leq \sup_{x \geq 1} 
\Big(\frac{\log(\e x)}{\sqrt{x}}\Big)^\delta = 
\Big(\frac{4}{\e}\Big)^{\delta/2} \lesssim 1,
\]
as needed. 
\end{proof}

\subsection{Upper bound in Gauss space.}
We will deduce the Gaussian estimate from the spherical estimate by integrating in
polar coordinates. That is, for $G\sim\gamma_n$, we write
\[
G=R \theta, \quad \mbox{where} \quad 
R \sim \chi_n, \theta = \frac{G}{\|G\|_2} \sim \sigma_{n-1},
\]
and, moreover, $R$ and
$\theta$ are independent.

For $C \in \AllConvSets{n}$ and $1\leq p <\infty$, define 
\[
m_{C,p}(r)= \Big(\E_{\theta \sim \sigma_{n-1}}
\|\Proj{C}(r\theta)\|_2^p\Big)^{1/p}, \quad \mbox{for}~r \geq 0.
\]
By \Cref{lem:pointwise-monotonicity-proj}, for every $1 \leq p <
\infty$, the averaged functions
$m_{C,p}$
and
$r \mapsto m_{C,p}(r)/r$ are nondecreasing and nonincreasing, respectively, for
$r \in (0, \infty)$. For such functions, we can identify the order of the
$L^r(\chi_n)$ norm. 
 
\begin{lemma}
\label{lem:radial-norm-comparison}
Let $m \colon \R_+ \to \R_+$ be nondecreasing and
suppose
that $t \mapsto m(t)/t$ is nonincreasing for $t > 0$. Then, for every
$1\leq r<\infty$,
\[
\|m\|_{L^r(\chi_n)}\asymp m(\|\chi_n\|_{L^r}).
\]
\end{lemma}

\begin{proof}
If $s \leq \|\chi_n\|_{L^r}$, then $m(s) 
\leq m(\|\chi_n\|_{L^r})$. On the other hand, 
if $\|\chi_n\|_{L^r} < s$, then $m(s)\leq \tfrac{s}{\|\chi_n\|_{L^r}}
m(\|\chi_n\|_{L^r})$. Combining these bounds,
\[
m(s)\leq m\big(\|\chi_n\|_{L^r}\big)\Big(1+\frac
{s}{\|\chi_n\|_{L^r}}\Big), \quad \mbox{for any}~s > 0.
\]
Taking $s = \chi_n$, we obtain $\|m\|_{L^r(\chi_n)}\leq 2 \,
m(\|\chi_n\|_{L^r})$.

For the reverse inequality, we use the Paley--Zygmund inequality:
\[
\P\Big\{\chi_n \geq \frac{\|\chi_n\|_{L^r}}{2}\Big\}^{1/r}
\geq \Big(1 - \frac{1}{2^r}\Big)^{2/r} 
\Big(\frac{\|\chi_n\|_{L^r}}{\|\chi_n\|_{L^{2r}}}\Big)^2 \gtrsim 1.
\]
Above, we used~\Cref{lem:chi-n-concentration}. 
Since $t \mapsto \tfrac{m(t)}{t}$ is nonincreasing, 
we have $m(\tfrac{1}{2}\|\chi_n\|_{L^r})\geq \tfrac{1}{2}m(\|\chi_n\|_{L^r})$.
Therefore
\[
\|m\|_{L^r(\chi_n)} \geq m\Big(\frac{\|\chi_n\|_{L^r}}{2}\Big)
\P\Big\{\chi_n \geq \frac{\|\chi_n\|_{L^r}}{2}\Big\}^{1/r}
\gtrsim m\big(\|\chi_n\|_{L^r}\big)
\]
as required.
\end{proof}

We are now in a position to prove the upper bound
in~\Cref{thm:optimal-reverse-holder-Gauss}. 

\begin{proof}[Proof of~\Cref{thm:optimal-reverse-holder-Gauss} (upper bound)]
Integrating 
in polar coordinates and 
using \Cref{lem:radial-norm-comparison},
\begin{equation}
\label{ineq:gaussian-upper-first}
\|\Proj{C}\|_{L^q(\gamma_n)}
=\|m_{C,q}\|_{L^q(\chi_n)}
\asymp m_{C,q}\big(\|\chi_n\|_{L^q}\big).
\end{equation}
For every $t>0$, it holds that
\[
\Proj{C}(t\theta)=t\Proj{C/t}(\theta),
\quad \mbox{where} \quad C/t= \{x/t : x\in C\}.
\]
Therefore, by \Cref{thm:optimal-reverse-holder-sphere}, for every $t > 0$,
\begin{equation}
\label{ineq:gaussian-spherical-at-radius}
m_{C,q}(t)
\leq \ExtProbProj{p}{q}{\sigma_{n-1}}m_{C,p}(t).
\end{equation}
Using \eqref{ineq:gaussian-spherical-at-radius} with $t=\|\chi_n\|_{L^q}$ and
using
$\|\chi_n\|_{L^q} \geq \|\chi_n\|_{L^p}$, we obtain 
\begin{multline}
\label{ineq:upper-bound-from-q-to-p}
m_{C,q}\big(\|\chi_n\|_{L^q}\big)
\\\leq \ExtProbProj{p}{q}{\sigma_{n-1}} \, 
m_{C,p}\big(\|\chi_n\|_{L^q}\big) 
\leq
\bigg[\ExtProbProj{p}{q}{\sigma_{n-1}}
\frac{\|\chi_n\|_{L^q}}{\|\chi_n\|_{L^p}}\bigg] 
\, m_{C,p}\big(\|\chi_n\|_{L^p}\big).
\end{multline}
Above, the final inequality came from the monotonicity
of
$t\mapsto
m_{C,p}(t)/t$. 
A second application of \Cref{lem:radial-norm-comparison} gives
$m_{C,p}(\|\chi_n\|_{L^p})\asymp\|\Proj{C}\|_{L^p(\gamma_n)}$. Combining this
with
\cref{ineq:gaussian-upper-first,ineq:gaussian-spherical-at-radius,ineq:upper-bound-from-q-to-p}
and then
using
\Cref{lem:chi-n-concentration}
and the upper bound from \Cref{thm:optimal-reverse-holder-sphere} (as established above),
\[
\frac{\|\Proj{C}\|_{L^q(\gamma_n)}}{\|\Proj{C}\|_{L^p(\gamma_n)}}
\lesssim
\ExtProbProj{p}{q}{\sigma_{n-1}}
\frac{\|\chi_n\|_{L^q}}{\|\chi_n\|_{L^p}}
\lesssim
\sqrt{\frac{\twomax{q}{n}}{\twomax{p}{n}}}
\Big(\frac{n}{(\twomin{p}{n})\log(\e \tfrac{n}{\twomin{q}{n}})}\Big)^\tau.
\]
Passing to the supremum over $C \in \AllConvSets{n}$ with $C \neq \{0\}$ yields the
result. 
\end{proof}

\section{Proofs of lower bounds}
\label{sec:lower-bounds} 

In this section, we present the proofs of the lower bounds of our main extremal results,~\Cref{thm:optimal-reverse-holder-sphere,thm:optimal-reverse-holder-Gauss}. 

\subsection{Supporting vectors of convex sets}
To obtain the lower bound for $\ExtProbProj{p}{q}{\sigma_{n-1}}$, in the regime
that $1 \leq p \leq q \leq p \log(\e n/p)$, we will use 
\emph{minimal norm supporting vectors}, as
introduced in~\Cref{sec:bounds-on-the-expected-norm}. 

For a probability measure $\mu$ on $\R^n$, define 
\[
\ExtProbSuppVec{p}{q}{\mu} = \sup_{\substack{C \in \AllConvSets{n}, C\neq \{0\} \\
C~\text{bounded}}}
\, \frac{\|\SuppVec{C}\|_{L^q(\mu)}}{\|\SuppVec{C}\|_{L^p(\mu)}}.
\]
The limit relation in~\Cref{lem:ray-limit-of-projections}
enables us to pass from the variational problem for projections, $\ExtProbProj{p}{q}{\sigma_{n-1}}$, to
the variational problem for supporting vectors, $\ExtProbSuppVec{p}{q}{\sigma_{n-1}}$. We will show, for all $1 \leq p \leq q \leq p \log(\e n/p)$, 
\begin{equation*}
\ExtProbSuppVec{p}{q}{\sigma_{n-1}} \asymp \ExtProbProj{p}{q}{\sigma_{n-1}}.
\end{equation*}
In words, the failure of dimension-free reverse Hölder inequalities, in
this range, is due to the supporting vectors of compact
convex sets. 

\medskip 

Throughout the remainder of this section, we use the notation 
\[
F_n(t) = \P_{\theta \sim \sigma_{n-1}}\{\theta_1 > t\}, 
\]
for $n \geq 1$ and $t > 0$. 
To obtain the lower bound on $\ExtProbSuppVec{p}{q}{\sigma_{n-1}}$, and thus on $\ExtProbProj{p}{q}{\sigma_{n-1}}$,
we consider the set
\[
C^\star_{n,p} = \conv\big(a_{n,p}B^n_2, e_1)
\quad \mbox{where} \quad 
F_n(a_{n,p}) = a_{n,p}^p.
\]

We now characterize the distribution of $\SuppVec{C^\star_{n,p}}$ and
the order of $a_{n,p}^p$.
\begin{lemma}[Properties of $C_{n,p}^\star$]
\label{lem:properties-of-bad-c}
For $1 \leq p \leq n$, the following hold. 
\begin{enumerate}[label=(\roman*)]
\item 
\label{item:well-defined}
The quantity $a_{n,p}$ is well-defined.
\item 
\label{item:asymptotics-of-anp}
It holds that 
\begin{equation}
\label{ineq:asymptotics-of-anp}
a_{n,p} \asymp \sqrt{\frac{p \log(\e n/p)}{n}}.
\end{equation}
\item 
\label{item:dist-of-support-bad-c}
When $\theta \sim \sigma_{n-1}$, the norm
$\|\SuppVec{C^\star_{n,p}}(\theta)\|_2$
has a two-point law: 
\[
a_{n,p}^p \delta_1 + (1 - a_{n,p}^p) \delta_{a_{n,p}}.
\]
\end{enumerate}
\end{lemma}
\begin{proof}
For $n \geq 1$ and $t \in [0, 1]$, we use the shorthand notation:
\[
F_n(t) = \P_{\theta \sim \sigma_{n-1}}\{\theta_1 > t\}.
\]
For $n = 1$, it is easy to see that $a_{1,1} = 1/2$. For $n \geq 2$, $F_n$
is continuous and decreasing, while $t \mapsto t^p$ is continuous and
increasing. Note that $F_n(0) = 1/2$ while $F_n(1) = 0$; hence there is a
unique $a_{n,p} \in (0,1)$ such that $F_n(a_{n,p})=a_{n,p}^p$, which yields
claim~\ref{item:well-defined}.


Turning to claim~\ref{item:asymptotics-of-anp}, we will show 
\[
\frac{1}{100}\, b\leq a_{n,p}\leq 240\,b 
\quad \mbox{where} \quad b \equiv b_{n,p} = \sqrt{\frac{p\log(\e n/p)}{n}}.
\]
The case $n=1$ follows from
$a_{1,1}=1/2$ and $b_{1,1}=1$. If $n=2$, then the bounds can be directly 
verified from 
$F_2(t)=\pi^{-1}\arccos(t)$. We henceforth assume
$n\geq 3$. Set 
\[
r=\frac{p}{n}, \quad \mbox{and} \quad L = \log\frac{\e}{r}. 
\] 
Note that with this notation, $b=\sqrt{rL}$.
Set 
\[
c = \frac{1}{100} \quad \mbox{and} \quad \delta = 12c^2 = \frac{3}{2500}.
\]
Since $cb<1/12$, \Cref{lem:spherical-estimates} gives
\[
F_n(cb)^{1/p}
\geq 12^{-1/p}\exp(-12nc^2b^2/p)
\geq \frac1{12}\exp(-\delta L)
=\frac{\e^{-\delta}}{12}r^\delta.
\]
On the other hand, since $1-2\delta>1/2$, it holds that $r^{1-2\delta}L\leq
\log(\e/r)\sqrt{r} \leq 2$, as $r \in (0, 1]$. 
Therefore, $\sqrt{rL}\leq 2r^\delta$, and 
\[
cb \leq 2cr^\delta=\frac1{50}r^\delta
<\frac{\e^{-\delta}}{12}r^\delta
\leq F_n(cb)^{1/p}. 
\]
Hence, we conclude $a_{n,p}\geq cb = b/100$.

For the upper bound, if $20b\geq 1/12$, then
$a_{n,p}<1\leq 240b$. Otherwise $20b<1/12$, and
\Cref{lem:spherical-estimates} gives
\[
F_n(20b)^{1/p}
\leq 12^{1/p}\exp(-n(20b)^2/(12p))
\leq 12\e^{-100L/3}
=12\Big(\frac{r}{\e}\Big)^{100/3}.
\]
Since $r\leq 1$ and $L\geq 1$, we have
\[
12\Big(\frac{r}{\e}\Big)^{100/3}
\leq 12\frac{r^{1/2}}{\e^{100/3}}
<20\sqrt{rL}
=20b.
\]
Thus $F_n(20b)^{1/p}<20b$, and therefore $a_{n,p}\leq 20b$ in this case.
Combining the cases gives $a_{n,p}\leq 240b$.

Finally, note that for any $x \in \R^n$, it holds that 
\[
h_{C^\star_{n,p}}(x)=\max\big\{x_1, a_{n,p} \|x\|_2\big\}.
\]
Consequently, for $\theta \in \bfS^{n-1}$ we have $h_{C^\star_{n,p}}(\theta) =
\max\{a_{n,p}, \theta_1\}$. It follows that 
$\SuppVec{C^\star_{n,p}}(\theta) = e_1$ if $\theta_1 > a_{n,p}$ and otherwise 
$\SuppVec{C^\star_{n,p}}(\theta) = a_{n,p} \theta$. 
The distribution of the norm $\|\SuppVec{C^\star_{n,p}}(\theta)\|_2$ now follows by
the definition of $a_{n,p}$ and the fact that the event $\{\theta_1 =
a_{n,p}\}$ is null.
\end{proof}

We now establish the lower bound.

\begin{proposition}
\label{prop:auxiliary-problem}
For every $n \geq 1$ and any $1 \leq p \leq q \leq \infty$, the following
hold.
\begin{enumerate}[label=(\roman*)]
\item 
\label{item:lower-bound-via-supporting-vector}
For any probability measure $\mu$ on $\R^n$, 
$\ExtProbProj{p}{q}{\mu}\geq \ExtProbSuppVec{p}{q}{\mu}$. 
\item 
\label{item:lower-bound-on-spherical}
It holds that 
\[
\ExtProbSuppVec{p}{q}{\sigma_{n-1}}
\gtrsim \Big(
\frac{n}{(\twomin{p}{n}) \log(\e \tfrac{n}{\twomin{p}{n}})}
\Big)^{\tau},
\]
where $\tau = \tfrac{1}{2}(1-\tfrac{p}{q})$.
\end{enumerate}
\end{proposition}
\begin{proof}
Throughout, we may assume $p < q$; otherwise, the result is trivial.
Denote by $\AllConvSets{n}'$ the class of bounded convex sets $C \in \AllConvSets{n}$ such that $C
\neq
\{0\}$. Note that for $1 \leq r < \infty$, 
\begin{equation}
\label{eqn:limit-relation-for-supporting-vector}
\lim_{\alpha \to 0^+} 
\frac{\E_\mu\|\Pi_{\alpha C}(X)\|_2^r}{\alpha^r}
= 
\lim_{\lambda \to \infty} 
\E_\mu\|\Pi_{C}(\lambda X)\|_2^r = \E_\mu \|\SuppVec{C}(X)\|_2^r,
\end{equation}
by dominated convergence, using, \eg that $\|\Pi_C(x)\|_2 \leq \rad_2(C) <
\infty$ and the pointwise convergence 
from~\Cref{lem:ray-limit-of-projections}.
Equivalently, for any $C \in \AllConvSets{n}'$ and any probability measure $\mu$ on
$\R^n$, 
\[
\|\Proj{\alpha C}\|_{L^r(\mu)} = 
\big(1 + o(1)\big)\, \alpha \|\SuppVec{C}\|_{L^r(\mu)}, \quad
\mbox{as}\quad
\alpha \to
0^+.
\]
Consequently, for any $1 \leq p < q < \infty$,
\begin{subequations}
\begin{equation}
\label{ineq:for-p-finite}
\ExtProbProj{p}{q}{\mu} \geq \sup_{\alpha > 0} \frac{\|\Pi_{\alpha
C}\|_{L^q(\mu)}}{\|\Pi_{\alpha
C}\|_{L^p(\mu)}} \geq 
\frac{\|\SuppVec{C}\|_{L^q(\mu)}}{\|\SuppVec{C}\|_{L^p(\mu)}}. 
\end{equation}
On the other hand, for $q = \infty$, we have from the above inequality  
\begin{equation}
\label{ineq:for-p-infinite}
\ExtProbProj{p}{\infty}{\mu} \geq \sup_{1 \leq q < \infty} \ExtProbProj{p}{q}{\mu} \geq 
\frac{\sup_{1 \leq q < \infty} \|\SuppVec{C}\|_{L^q(\mu)}}{\|\SuppVec{C}\|_{L^p(\mu)}} = 
 \frac{\|\SuppVec{C}\|_{L^\infty(\mu)}}{\|\SuppVec{C}\|_{L^p(\mu)}}.
\end{equation}
\end{subequations}
Passing to the supremum over $C \in \AllConvSets{n}'$ in
inequalities~\eqref{ineq:for-p-finite} and~\eqref{ineq:for-p-infinite} yields
$\ExtProbProj{p}{q}{\mu}  \geq \ExtProbSuppVec{p}{q}{\mu}$ for any $1 \leq p < q \leq
\infty$, which establishes claim~\ref{item:lower-bound-via-supporting-vector}.

For claim~\ref{item:lower-bound-on-spherical},
as $\ExtProbSuppVec{p}{q}{\mu} \geq 1$, it clearly suffices to consider
the case when $1 \leq p \leq n$. Note that for any $r \geq p \geq 1$,
by~\Cref{lem:properties-of-bad-c}\ref{item:dist-of-support-bad-c}, it holds
that
\[
\|\SuppVec{C^\star_{n,p}}\|_{L^r(\sigma_{n-1})}
= 
\big(a_{n,p}^p + (1-a_{n,p}^p) a_{n,p}^r\big)^{1/r}
\asymp a_{n,p}^{p/r},
\] 
as $a_{n,p} \in (0, 1]$ by definition. Consequently, taking $r \in \{p, q\}$
yields:

\[
\ExtProbSuppVec{p}{q}{\sigma_{n-1}} \geq 
\frac{\|\SuppVec{C^\star_{n,p}}\|_{L^q(\sigma_{n-1})}}{\|\SuppVec{C^\star_{n,p}}\|_{L^p
(\sigma_{n-1})}} \asymp a_{n,p}^{-2\tau} 
\asymp 
\Big(\frac{n}{p \log (\e n/p)} \Big)^\tau,
\]
as required. 
\end{proof}

\subsection{Lower bound for the sphere} 
\label{sec:lower-bound-in-the-spherical-case}

To establish the lower bound in~\Cref{thm:optimal-reverse-holder-sphere}, we
consider three regimes, depending on the configuration of the triple $(n, p,
q)$. 

\medskip 

Throughout this section, let us denote the lower bound 
\[
\ell(n,p,q) = \Big(\frac{n}{(\twomin{p}{n})\log(\e 
\tfrac{n}{\twomin{q}{n}})}\Big)^{\tfrac{1}{2}(1 -
\tfrac{p}{q})}.
\]
In the construction of the lower bound, we consider the sets 
\[
C(a, b) = b \conv(a B^n_2, e_1),
\]
for $a, b > 0$. We set $c = \frac{1}{144(1+\sqrt{6})^2}$.

\subsubsection{Case 1, $p \leq q \leq p \log(\e n/p)$.}
\label{sec:small-q-case}

A simple observation is that in this case, we have 
\begin{equation}
\label{eqn:equivalent-for-ell-small-q} 
\ell(n, p, q) \asymp \Big(\frac{n}{p \log (\e n/p)}\Big)^{\tfrac{1}{2}(1 -
\tfrac{p}{q})}.
\end{equation}
Indeed, this follows from 
\[
\frac{1}{\e} \log \frac{\e n}{p} \leq \log \frac{\e n}{p} + \log \frac{p}{q} =
\log
\frac{\e
n}{q}
\leq \log \frac{\e n}{p}.
\]
The last inequality used $q \leq p \log(\e n /p)$ and $\log
\frac{x}{\log x} \geq \frac{1}{\e} \log x$ for $x \geq \e$, which was applied
with $x = \tfrac{\e n}{p}$. 
Now by~\Cref{prop:auxiliary-problem} and using $\twomin{p}{n} = p$,
\[
\ExtProbProj{p}{q}{\sigma_{n-1}} \geq \ExtProbSuppVec{p}{q}{\sigma_{n-1}} \gtrsim
\Big(\frac{n}{p \log(\e n/p)}\Big)^{\tfrac{1}{2}(1-\tfrac{p}{q})} 
\asymp \ell(n,p,q),
\]
by~\cref{eqn:equivalent-for-ell-small-q}.

\subsubsection{Case 2, $p \log(\e n/p) < q < c\, n$.}
\label{sec:large-q-case}

We start with the basic observation that in this case, 
\begin{equation}
\label{eqn:equivalent-in-the-middle-case}    
\ell(n, p, q) \asymp \sqrt{\frac{n}{p \log(\e n/q)}}.
\end{equation}
Indeed, this follows from the observation that 
\begin{equation}
\label{ineq:equivalent-in-middle-v2}    
1 \leq \Big(\frac{n}{p \log(\e n/q)}\Big)^{p/q} \leq
\Big(\frac{n}{p}\Big)^{p/q} = 
\exp\Big(\frac{p}{q} \log \frac{n}{p}\Big) \leq \e.
\end{equation}

To establish the lower bound, we need the following result, which characterizes
the projection onto the sets $C(a, b)$. 
\begin{lemma}
\label{lem:projection-norm-for-C}
Let $n \geq 2$. 
For any $a, b \in (0, 1]$ and any $\theta \in \bfS^{n-1}$, it holds
that
\begin{equation}
\label{eqn:formula-of-the-norm-projection}
\|\Proj{C(a,b)}(\theta)\|_2 = \sqrt{a^2 b^2 + \tau^2(\theta)}, 
\end{equation}
where 
\[
\tau(\theta)^2 = \min\bigg\{b^2(1 -a^2), \Big(\theta_1
\sqrt{1-a^2} - a
\sqrt{1 - \theta_1^2}\Big)_+^2\bigg\}.
\]    
In particular, if $a < 1$, then it holds that 
\begin{equation}
\label{eqn:key-equivalent-for-the-projection}
\|\Proj{C(a,b)}(\theta)\|_2 \asymp a b+
\frac{1}{\sqrt{1-a^2}}\min\Big\{b(1-a^2), (\theta_1 -
a)_+\Big\}. 
\end{equation}
\end{lemma}
\begin{proof}

We compute the projection according to the location of $\theta_1$.
Throughout we set 
\[
A \equiv A(a,b) =  b(1-a^2) + a \sqrt{1-b^2(1-a^2)}.
\]
\paragraph{Case I:  $\theta_1\le a$}
In this case, 
\[
h_{C(a,b)}(\theta)=b\max\{a,\theta_1\}=ab.
\] 
Consequently, for every
$z\in C(a,b)$,
\[
\langle\theta-ab\theta,z-ab\theta\rangle
=(1-ab)(\langle\theta,z\rangle-ab)
\leq 0.
\]
\Cref{lem:projection-variational} then implies $\Proj{C(a,b)}(\theta)=ab\theta$. In this case, 
by direct computation, $\tau^2(\theta) = 0$, and hence we obtain~\eqref{eqn:formula-of-the-norm-projection}.

\medskip
\paragraph{Case II: $a < \theta_1 < A$} Equivalently, we have 
\[
0 < \theta_1\sqrt{1-a^2}-a\sqrt{1-\theta_1^2}
< b\sqrt{1-a^2}.
\]
There exists a unit vector $u\in\bfS^{n-1}\cap e_1^\perp$ such that 
\[
\theta=\theta_1e_1+\sqrt{1-\theta_1^2}\,u.
\]
Consider the vectors
\[
w=ae_1+\sqrt{1-a^2}\,u,\quad \mbox{and} \quad 
v=\sqrt{1-a^2}\,e_1-au.
\]
Then $w$ and $v$ are orthonormal; moreover, we can write
\[
\theta= \alpha w+ \beta v, 
\quad e_1 = aw +\sqrt{1-a^2} v.
\]
Above, 
\[
\alpha=a\theta_1+\sqrt{1-a^2}\sqrt{1-\theta_1^2}, 
\quad \beta = \theta_1\sqrt{1-a^2}
-a\sqrt{1-\theta_1^2}.
\]
We write $y = ab w + \beta v$. Since $be_1 = abw + b\sqrt{1-a^2}\, v$ and $\beta \in (0, b \sqrt{1-a^2})$, we see that $y \in \conv(\{ab \, w, be_1\})\subset C(a,b)$. Moreover,  
\[
\langle \theta - y, be_1 - y \rangle = 
(\alpha - ab) (b\sqrt{1-a^2} - \beta) \langle  w, v \rangle = 0.
\]
On the other hand, 
\[
ab \|\theta - y\|_2 -
\langle \theta - y, y \rangle 
= ab \Big(|\alpha - ab| - (\alpha - ab)\Big) = 0.
\]
Above, we used 
\[
\alpha - ab = \sqrt{1-\beta^2} - ab 
> \sqrt{a^2 b^2 + 1 - b^2} - ab \geq 0.
\]
Therefore, by convexity $\langle \theta - y, z - y\rangle \leq 0$ for all $z \in C(a,b)$, and hence $\Proj{C(a,b)}(\theta) = y$. Note that 
\[
\|\Proj{C(a,b)}(\theta)\|_2 = \sqrt{a^2 b^2 + \gamma^2(\theta)}, 
\]
where 
\[
\gamma^2(\theta) = \beta^2 = (\theta_1\sqrt{1-a^2} - a\sqrt{1-\theta_1^2})^2.
\]
Because $\theta_1 \leq A$, $\gamma^2(\theta) = \tau^2(\theta)$, 
as required. 

\medskip 
\paragraph{Case III: $A \leq \theta_1 \leq 1$} 
Equivalently, we have 
\[
(\theta_1 - b)\sqrt{1-a^2}\geq a\sqrt{1-\theta_1^2}.
\]
Squaring this inequality,
\[
(\theta_1-b)^2 \geq a^2\Big(1-\theta_1^2 + (\theta_1-b)^2\Big).
\]
If $z \in ab B^n_2$, then 
\[
\langle \theta - be_1, z - be_1\rangle 
\leq ab \|\theta - b e_1\|_2 - b(\theta_1 - b)
\leq 0.
\]
By convexity and~\Cref{lem:projection-variational}, this implies $\Proj{C(a,b)}(\theta) = be_1$.
Since $\theta_1 \geq A$, 
it holds that $\tau^2(\theta) = b^2(1-a^2)$ and thus we obtain~\eqref{eqn:formula-of-the-norm-projection}.

\medskip
\paragraph{Proof of~\eqref{eqn:key-equivalent-for-the-projection}} Assume that $a<1$.
First observe that: 
\[
\theta_1\sqrt{1-a^2}-a\sqrt{1-\theta_1^2}
=\frac{(\theta_1-a)(\theta_1+a)}
{\theta_1\sqrt{1-a^2}+a\sqrt{1-\theta_1^2}}.
\]
First suppose $\theta_1>a$. Since
$\sqrt{1-\theta_1^2}\le\sqrt{1-a^2}$, the display above shows
\[
\qquad
\frac{\theta_1-a}{\sqrt{1-a^2}}
\le \theta_1\sqrt{1-a^2}-a\sqrt{1-\theta_1^2}
\le\frac{2(\theta_1-a)}{\sqrt{1-a^2}}.
\]
Therefore,
\begin{multline*}
\min\Big\{b\sqrt{1-a^2},
\big(\theta_1\sqrt{1-a^2}-a\sqrt{1-\theta_1^2}\big)_+\Big\}
\\\asymp \frac1{\sqrt{1-a^2}}\min\{b(1-a^2),(\theta_1-a)_+\}.
\end{multline*}
Of course, if $\theta_1 < a$, then the relation above continues to hold (both sides vanish). The conclusion then follows from $\sqrt{x^2+y^2}\asymp x+y$, for $x, y \geq 0$.
\end{proof}

To obtain the lower bound in this case, 
we
consider
the
set
\[
C_n = C(a_n, b_n) \quad \mbox{where} \quad a_n = 
\sqrt{6\frac{p}{n} \log
\frac{\e
n}{q}}
\quad \mbox{and} \quad 
b_n =  \sqrt{\frac{q}{n}}. 
\]
(Of course, $a_n, b_n$ also depend on the pair $(p, q)$, but we suppress that in
the notation.)
Observe that $a_n^2 \in (0, 6]$. 
If $a_n > 1/2$, then 
\[
\ExtProbProj{p}{q}{\sigma_{n-1}} \geq 1 \asymp \frac{1}{a_n} \asymp \ell(n,p,q).
\]
On the other hand, if $a_n^2 \in (0, 1/2]$, then $(1-a_n^2) \asymp
1$,
and
hence
\begin{equation}
\label{eqn:key-relation-for-the-lower-bound-set}
\|\Proj{C_n}(\theta)\|_2 
\asymp a_n b_n + \min\{b_n, (\theta_1 - a_n)_+\}, \quad \mbox{for any}~\theta
\in \bfS^{n-1},
\end{equation}
by~\Cref{lem:projection-norm-for-C}.
Consequently, by~\Cref{lem:truncated-lp-estimate}, 
\begin{align}
\|\Proj{C_n}\|_{L^p(\sigma_{n-1})} 
&\lesssim 
a_n b_n + 
\|(\theta_1 - a_n)_+\|_{L^p}
\nonumber \\ 
&\lesssim 
a_n b_n + \frac{p}{n \, a_n}  \sqrt\frac{q}{
n}
\asymp \Big(1 + \frac{1}{\log(\e n/q)}\Big) a_n b_n
\lesssim a_n b_n.
\label{ineq:upper-on-p-norm-middle-case}
\end{align} 
Additionally, from
relation~\eqref{eqn:key-relation-for-the-lower-bound-set}, we have
\[
\|\Proj{C_n}\|_{L^q(\sigma_{n-1})} \gtrsim b_n \P\{\theta_1 \geq a_n + b_n\}^{1
/q}.
\]
On the other hand, using $q \geq p \log \tfrac{\e n}{p}$, we have 
\[
\frac{a_n^2}{6} = b_n^2  \frac{p}{q} \log \frac{\e n}{q}
\leq b_n^2.
\]
Hence, $a_n \leq \sqrt{6} b_n$, and consequently $a_n + b_n \leq (1+\sqrt{6}) \sqrt{\tfrac{q}{n}} < \tfrac{1}{12}$ by our choice of $c$. 
Combining the previous two displays, and using
\Cref{lem:spherical-estimates}\ref{ineq:lower-bound-on-lower-tail}, 
\begin{equation}
\label{ineq:lower-on-q-norm-middle-case}
\|\Proj{C_n}\|_{L^q(\sigma_{n-1})}
\gtrsim b_n \e^{-12(1+\sqrt{6})^2} \asymp b_n.
\end{equation}
Combining inequalities~\eqref{ineq:upper-on-p-norm-middle-case}
and~\eqref{ineq:lower-on-q-norm-middle-case}, we obtain 
\[
\ExtProbProj{p}{q}{\sigma_{n-1}}
\geq \frac{\|\Proj{C_n}\|_{L^q(\sigma_{n-1})}}{\|\Proj
{C
_n
}\|_{L^p
(\sigma_{n-1
})}} 
\gtrsim \frac{1}{a_n} \asymp \ell(n,p,q),
\]
by display~\eqref{eqn:equivalent-in-the-middle-case}, as required. 

\subsubsection{Case 3, $q \geq c\, n$ and $q > p \log \tfrac{\e n}{p}$.} 
In this case, we observe that 
\begin{equation}
\label{eqn:equivalent-in-large-case}
\ell(n,p,q) \asymp \max\Big\{1, \sqrt{\frac{n}{p}}\Big\}.    
\end{equation}
Indeed, if $p \geq n$, then $\ell(n,p,q) = 1$, so there is nothing to prove. On
the other hand, if $p \leq n$, then the above display follows from
\[
1 \leq \Big(\frac{n}{p}\Big)^{p/q}  \leq \e, 
\]
with the last inequality following from our assumption on the triple $(p, q, n)$.

Now, we consider the 
ray and its associated projector,
\[
C= \{\lambda e_1 : \lambda \geq 0\}, \quad  
\mbox{and} \quad 
\Proj{C}(x)=(x_1)_+e_1.
\]
Hence, with $\theta \sim \sigma_{n-1}$, it holds
by~\Cref{lem:coordinate-moments} that
\begin{multline*}
\ExtProbProj{p}{q}{\sigma_{n-1}}
\\
\geq \frac{\|\Proj{C}\|_{L^q(\sigma_{n-1})}}{\|\Proj{C}\|_{L^p
(\sigma
_
{n
-1})}}
= 
\frac{\|(\theta_1)_+\|_{L^q}}{\|(\theta_1)_+\|_{L^p}} \asymp
\sqrt\frac{\twomin{q}{n}}{\twomin{p}{n}} \asymp \max\Big\{1, \sqrt{\frac{n}{p}}\Big\} \asymp
\ell(n,p,q),
\end{multline*}
by relation~\eqref{eqn:equivalent-in-large-case}.
\subsection{Lower bound in Gauss space}

In this section, we establish the lower bound when $\mu = \gamma_n$.

\begin{proof}[Proof of~\Cref{thm:optimal-reverse-holder-Gauss} (lower bound)]

We split the argument into three cases, 
depending on the configuration of the
triple $(n, p, q)$. 

\medskip
\paragraph{Case 1: $q\leq n$}
In this case, we can reduce the problem to the spherical setting. 
Fix a convex set $C \in\AllConvSets{n}$ such that $C \neq\{0\}$. 
Since 
$p, q \leq n$, \Cref{lem:chi-n-concentration} gives 
$\|\chi_n\|_{L^p} \asymp \|\chi_n\|_{L^q} \asymp \sqrt{n}$. Hence, for $r
\in \{p, q\}$, by
\Cref{lem:radial-norm-comparison}
and the monotonicity of
$m_{C,r}$,
\[
\|\Proj{\sqrt{n} C}\|_{L^r(\gamma_n)}
= 
\|m_{\sqrt{n} C,r}\|_{L^r(\chi_n)}
\asymp m_{\sqrt{n} C,r}\Big(\|\chi_n\|_{L^r}\Big)
\asymp m_{\sqrt{n} C,r}(\sqrt n).
\]
On the other hand, 
\[
m_{\sqrt{n} C,r}(\sqrt n) = \sqrt{n}\, 
m_{C, r}(1)
=\sqrt n\,\|\Proj{C}\|_{L^r(\sigma_{n-1})}.
\]
Consequently, combining the previous two displays, we have 
\[
\frac{\|\Proj{\sqrt{n}
C}\|_{L^q(\gamma_n)}}{\|\Proj{\sqrt{n} C}\|_{L^p(\gamma_n)}}
\asymp
\frac{\|\Proj{C}\|_{L^q(\sigma_{n-1})}}{\|\Proj{C}\|_{L^p(\sigma_{n-1})}}.
\]
Passing to the supremum over $C \in \AllConvSets{n}$ with $C \neq \{0\}$, we obtain 
\[
\ExtProbProj{p}{q}{\gamma_n}
\asymp 
\ExtProbProj{p}{q}{\sigma_{n-1}}
\asymp 
\sqrt{\frac{\twomax{q}{n}}{\twomax{p}{n}}}\Big(\frac{n}{(\twomin{p}{n})\log(\e \tfrac{n}{\twomin{q}{n}})}\Big)^\tau,
\]
where we applied 
\Cref{thm:optimal-reverse-holder-sphere}, 
and the fact that $p, q \leq n$. 

\medskip
\paragraph{Case 2: $p\geq n$}
Take $C=\R^n$. 
Then, by
\Cref{lem:chi-n-concentration} and the relation ${\twomin{q}{n} = \twomin{p}{n} = n}$, we
obtain:
\[
\ExtProbProj{p}{q}{\gamma_n}
\geq \frac{\|\chi_n\|_{L^q}}{\|\chi_n\|_{L^p}}
\asymp \sqrt{\frac{\twomax{q}{n}}{\twomax{p}{n}}}
\asymp 
\sqrt{\frac{\twomax{q}{n}}{\twomax{p}{n}}} \bigg(\frac{n}{(\twomin{p}{n}) \log(\e \tfrac{n}{\twomin{q}{n}})}\bigg)^\tau,
\]
as required. 

\medskip
\paragraph{Case 3: $p<n<q$}
We consider the ray and its associated projector, 
\[
C= \{\lambda e_1 : \lambda \geq 0\}, \quad  
\mbox{and} \quad 
\Proj{C}(x)=(x_1)_+e_1.
\]
Hence, $\|\Proj{C}(G)\|_2=(G_1)_+$, and thus 
standard Gaussian estimates yield
\[
\|(G_1)_+\|_{L^r} = 2^{-1/r} \|G_1\|_{L^r} 
\asymp \|G_1\|_{L^r} \asymp  \sqrt r,
\quad \mbox{for}~1\leq r < \infty.
\]
Thus
\[
\ExtProbProj{p}{q}{\gamma_n}
\geq
\frac{\|(G_1)_+\|_{L^q}}{\|(G_1)_+\|_{L^p}}
\asymp \sqrt{\frac{q}{p}} 
= \sqrt{B(n,p,q)} \sqrt{\frac{\twomax{q}{n}}{\twomax{p}{n}}} \bigg(\frac{n}{(\twomin{p}{n}) \log(\e
\tfrac{n}{\twomin{q}{n}})}\bigg)^\tau.
\]
Above, using $\twomax{q}{n} = q, \twomin{q}{n} = n, \twomin{p}{n} = p, \twomax{p}{n}=n$, we have
\begin{multline*}
B(n,p,q) = \frac{q}{\twomax{q}{n}}\frac{\twomax{p}{n}}{p}
\bigg(\frac{(\twomin{p}{n}) \log(\e
\tfrac{n}{\twomin{q}{n}})}{n}\bigg)^{1-p/q}
\\=
\bigg(\frac{n}{p}\bigg)^{p/q}
\geq \Big(\inf_{x \geq 1} x^{1/x}\Big)^{n/q} 
\geq 1.
\end{multline*}
Combining the previous two displays yields the claim. 
\end{proof}
\section{Proofs from~\Cref{sec:bounds-on-the-expected-norm}}
\label{sec:deferred-proofs-of-upper-bounds-on-expected-norm}

In this section, we collect the proofs of the upper bounds on the norm of the projection, which were developed in~\Cref{sec:bounds-on-the-expected-norm}.

\subsection{Proof of~\Cref{prop:qualitative-behavior-of-metric-projections}}
Put $K=\cl\cone C$. The projection satisfies the rescaling relation,
\[
\Proj{\lambda C}(x)=\lambda \Proj{C}(x/\lambda), \qquad 
\mbox{for all}~x\in \R^n,~\mbox{and all}~\lambda > 0.
\]
Fix nonzero $x\in\R^n$. We can write $x=r\theta$ for $r>0$ and 
$\theta \in \bfS^{n-1}$. By
\Cref{lem:pointwise-monotonicity-proj},
\[
\|\Proj{\lambda C}(x)\|_2
= \lambda \|\Proj{C}(r\theta/\lambda)\|_2 = 
r\,\frac{\|\Proj{C}((r/\lambda)\theta)\|_2}{r/\lambda},
\]
which is nondecreasing in $\lambda$. We claim that its limit is
$\|\Proj{K}(x)\|_2$. Indeed, let
\[
\varphi_x(z)=2\langle x,z\rangle-\|z\|_2^2.
\]
Since $0\in C$, the sets $\lambda C$ are increasing and
\[
\bigcup_{\lambda>0}\lambda C=\cone C.
\]
Moreover, since $\varphi_x$ is continuous and $K=\cl\cone C$,
\begin{equation}
\label{eqn:convergence-of-suprema}
\sup_{z\in \lambda C}\varphi_x(z)
\uparrow
\sup_{z\in K}\varphi_x(z),
\qquad \mbox{as } \lambda\uparrow\infty.
\end{equation}
Indeed, the upper bound is immediate from $\lambda C\subset K$. Conversely,
if $u=\Proj{K}(x)$ and $\varepsilon>0$, then by the continuity of
$\varphi_x$ and since $\cone C$ is dense in $K$, there exists
$v\in\cone C$ such that
\[
\varphi_x(v)\geq \varphi_x(u)-\varepsilon.
\]
Since $v\in\cone C=\bigcup_{\lambda>0}\lambda C$, we have
$v\in\lambda_0 C$ for some $\lambda_0>0$, and hence
\[
\sup_{z\in\lambda C}\varphi_x(z)
\geq \varphi_x(v)
\geq \varphi_x(u)-\varepsilon
\qquad \mbox{for all } \lambda\geq \lambda_0.
\]
This establishes the convergence~\eqref{eqn:convergence-of-suprema}. Now put $u_\lambda=\Proj{\lambda C}(x)$ and $u=\Proj{K}(x)$. The variational
characterization of the projection gives
\[
\varphi_x(u_\lambda)
=
\sup_{z\in\lambda C}\varphi_x(z)
\to 
\sup_{z\in K}\varphi_x(z)
=
\varphi_x(u).
\]
Since $u_\lambda \in K$, the variational inequality for $u=\Proj{K}(x)$ yields
\[
\|u_\lambda-u\|_2^2
\leq
\|x-u_\lambda\|_2^2-\|x-u\|_2^2
=
\varphi_x(u)-\varphi_x(u_\lambda)
\to0.
\]
Thus
\[
\|\Proj{\lambda C}(x)\|_2
\uparrow
\|\Proj{K}(x)\|_2
\qquad \text{as} \quad \lambda \to \infty.
\]
For $1\leq p<\infty$, claim~\ref{item:cone-upper-bound-and-limit}
follows from the monotone convergence theorem applied to
$\|\Proj{\lambda C}(\cdot)\|_2^p$. For $p=\infty$, we use the elementary
fact that if $0\leq f_\lambda\uparrow f$ pointwise, then
$\|f_\lambda\|_{L^\infty(\mu)}\uparrow \|f\|_{L^\infty(\mu)}$.
On the other hand, the same scaling identity gives
\[
\frac{\|\Proj{\lambda C}(x)\|_2}{\lambda}
=
\|\Proj{C}(x/\lambda)\|_2 \uparrow \|\SuppVec{C}(x)\|_2, \quad 
\mbox{as}~\lambda \downarrow 0. 
\]
Here, we applied 
\Cref{lem:pointwise-monotonicity-proj,lem:ray-limit-of-projections}. In particular,
$\|\Proj{\lambda C}(x)\|_2\leq \lambda \|\SuppVec{C}(x)\|_2$ 
for every $\lambda>0$. Applying the same argument to
$f_\lambda(x)=\|\Proj{\lambda C}(x)\|_2/\lambda$ yields
claim~\ref{item:supporting-vector-upper-bound-and-limit}.

\subsection{Proof of~\Cref{prop:upper-bound-via-typical-radius}}
Claim~\ref{item:well-defined-fixed-point} follows from~\cite[Proposition 2.1]{PatZhi26}; hence we focus on establishing claim~\ref{item:upper-bound-on-mean-norm}.
Throughout, we denote by
\[
r_\star = r_\mu(C) \quad \mbox{and} \quad \hat r = 
\hat r(\xi) = \argmax_{r \geq 0} \Big\{\, h_{C \cap r B^n_2}(\xi) - \frac{r^2}{2}\, \Big\}.
\]
The crux of the proof is the following convex-analytic 
statement. 
\begin{lemma}
\label{lem:key-properties}
    There is a measurable map $s \colon \R^n \to \R_+$ such that 
    the following hold.
    \begin{enumerate}[label=(\roman*)]
    \item 
    \label{item:subgradient-inequality}
    For every $r > r_\star$ and $x \in \R^n$, it holds that
    \[
    h_{C \cap r B^n_2}(x) \leq h_{C \cap r_\star B^n_2}(x) + s(x) (r - r_\star).
    \]
    \item
    \label{item:optimality}
    In expectation, $\E s(\xi) \leq r_\star$.
    \end{enumerate}
\end{lemma}

\begin{proof}
We use the shorthand notation 
\[
\psi_x(r) = h_{C \cap r B^n_2}(x).
\]
By the convexity of $C$, the map $\psi_x$ is concave on $\R_+$ for any $x \in \R^n$. Recall the right derivative
\begin{equation*}
\RightDer \psi_x(r) 
= \lim_{s \downarrow r} \frac{\psi_x(s)-\psi_x(r)}{s - r} = 
\sup_{s > r} \frac{\psi_x(s)-\psi_x(r)}{s - r}.
\end{equation*}
The final equality holds by concavity of $\psi_x$ and the fact that the slopes are nonincreasing; see~\cite[Theorems 0.6.2 and 0.6.3]{HirLem01} for a formal statement. 
Now, we can set 
\[
s(x) = \RightDer \psi_x(r_\star).
\] 
By definition, we obtain Lemma~\ref{lem:key-properties}\ref{item:subgradient-inequality}. 
For the measurability, note that for any positive sequence $h_n \downarrow 0$, we may write
\[
s(x) = \lim_{n \to \infty} s_n(x), 
\quad \mbox{for} \quad s_n(x) = \frac{\psi_x(r_\star + h_n) - \psi_x(r_\star)}{h_n}.
\]
On the other hand, each $s_n$ is measurable: $x \mapsto \psi_{x}(r)$ is continuous since $h_{C\cap rB^n_2}(\cdot)$ is convex. Thus, $s$ is measurable. 
Finally, by the assumption that $r_\mu(C)$ is well-defined, the convexity of $C$ implies that 
$r \mapsto \Psi(r) = \E_\mu h_{C \cap r B^n_2}(\xi)$ is a finite, concave function on $\R_+$. As $r_\star$ maximizes $r \mapsto \Psi(r) - r^2/2$, we conclude
\[
r_\star \geq \RightDer\Psi(r_\star) 
= \lim_{n \to \infty } \E \frac{\psi_\xi(r_\star + h_n) - \psi_{\xi}(r_\star)}{h_n} 
= \lim_{n \to \infty} \E s_n(\xi) = \E s(\xi),
\]
since $s_n \uparrow s$ by concavity;~\Cref{lem:key-properties}\ref{item:optimality} follows.
\end{proof}

We are now in a position to complete the proof of~\Cref{prop:upper-bound-via-typical-radius}.

\begin{proof}[Proof of~\Cref{prop:upper-bound-via-typical-radius}]
Suppose that $\hat r > r_\star$, and set
$\psi(r)=h_{C\cap rB_2^n}(\xi)$. 
Since $\hat r$ maximizes the concave function
$r\mapsto \psi(r)-r^2/2$ and $\hat r>0$, the one-sided optimality
conditions give
\[
\RightDer \psi(\hat r)\leq \hat r\leq \mathsf D_- \psi(\hat r),
\]
where $\mathsf D_\pm \psi$ denotes the left and right derivatives. Since the one-sided
derivatives of a concave function are monotone nonincreasing and
$r_\star<\hat r$, we have
\[
s(\xi)=\RightDer\psi(r_\star)\geq \mathsf D_- \psi(\hat r)\geq \hat r.
\]
Consequently, combining the cases, we obtain
\[
\|\Pi_C(\xi)\|_2 \leq r_\star \1\{\hat r \leq r_\star\} + s(\xi) \1\{\hat r > r_\star\}. 
\]
Taking expectations and applying~\Cref{lem:key-properties}\ref{item:optimality} yields the claim.
\end{proof}

\subsection{Proof of~\Cref{thm:distance-from-typical-radius}}
\label{sec:distance-from-typical-radius}

In fact, we prove a more general result for Orlicz norms, which we now recall. 

\begin{definition}[Orlicz norm]
    Let $\psi \colon \R_+ \to \R_+$ be an Orlicz function (\ie a convex, increasing function with $\psi(0) = 0$ and $\lim_{x \to \infty} \psi(x) = \infty$). Then, the \emph{Orlicz norm of the random variable $X$} is given by
    \[
    \|X\|_\psi = \inf \Big\{\, t > 0 : \E \psi(|X|/t) \leq 1\,\Big\}.
    \]
\end{definition}

If $X \sim \mu$ and $f \colon \R^n \to \R$ is measurable, we write 
\[
\|f\|_{\psi(\mu)} = \|f(X)\|_{\psi} = 
\inf\Big\{t > 0 : \E_\mu \psi(|f(X)|/t) \leq 1\,\Big\}.
\]
We can also consider the Orlicz constant for convex, Lipschitz functionals: 
\[
K_\psi(\mu) = \sup\Big\{\,\|f - \E f\|_{\psi(\mu)} \mid 
f \colon \R^n \to \R,~\text{convex, $1$-Lipschitz}\,\Big\}.
\]
Throughout this section, we use the shorthand notation
\[
\begin{gathered}
\kappa=K_\psi(\mu),\quad r_\star=r_\mu(C),
\quad
\hat r=\|\Proj{C}(\xi)\|_2,
\quad 
Z_r=h_{C \cap r B^n_2}(\xi), \\
\phi_\xi(r)=h_{C \cap r B^n_2}(\xi)-\frac{r^2}{2},\qquad
\Phi(r)=\E Z_r-\frac{r^2}{2}.
\end{gathered}
\]
\begin{lemma}
\label{lem:tail-bounds-for-projection-norm}
Suppose that $C \in \AllConvSets{n}$ and that $\mu$ is a probability measure on $\R^n$ such that $r_\mu(C)$ and $K_\psi(\mu)$ are finite. Then, there exist universal constants $c_1, c_2 > 0$ such that for every $s \geq 1$,
\[
\P\Big\{|\hat r - r_\mu(C)| > A_\psi(C, \mu) \, s\Big\} \leq \frac{c_1}{\psi(c_2 s)},
\]
where 
$A_\psi(C, \mu) = \max\Big\{K_\psi(\mu), \sqrt{K_\psi(\mu) r_\mu(C)}\Big\}$. 
\end{lemma}
\begin{proof}
Recall that by \Cref{lem:radius-interpretation-of-projection}, $\hat r$ uniquely maximizes the $1$-strongly concave function $r \mapsto \phi_\xi(r)$ on $\R_+$. Hence, 
if $t > r_\star$, then on the event 
$\{\hat r \geq t\}$, the concave function
$r\mapsto \phi_\xi(r)$ is nondecreasing on $[0,t]$, and therefore
$\phi_\xi(t)\geq \phi_\xi(r_\star)$. Therefore,
\begin{multline}
\label{eqn:key-inequality-on-centered-support}
(Z_t-\E Z_t)-(Z_{r_\star}-\E Z_{r_\star}) \\= 
(\phi_\xi(t) - \phi_{\xi}(r_\star)) 
- (\Phi(t) - \Phi(r_\star)) 
\geq
\Phi(r_\star)-\Phi(t)
\geq
\frac{(t-r_\star)^2}{2},
\end{multline}
since $\Phi$ is $1$-strongly concave, where we again applied \Cref{lem:radius-interpretation-of-projection}. 
The map $\xi \mapsto h_{C \cap rB^n_2}(\xi) \equiv Z_r$ is convex and $r$-Lipschitz. Hence, for any $s, t > 0$,
\[
\|(Z_t-\E Z_t)-(Z_{s}-\E Z_{s})\|_{\psi(\mu)} 
\leq \kappa (s+t).\] 
Therefore, by Markov's inequality 
\begin{multline*}
\P\{\hat r \geq t\} 
\leq 
\P\Big\{(Z_t-\E Z_t)-(Z_{r_\star}-\E Z_{r_\star}) 
\geq \frac{(t- r_\star)^2}{2} \Big\}
\leq \frac{1}{\psi\Big(\frac{1}{2}\frac{(t-r_\star)^2}{\kappa (t+r_\star)}\Big)}.
\end{multline*}
On the other hand, if $t < r_\star$, then note that on $\{\hat r \leq t<r_\star\}$, the map 
$r \mapsto \phi_\xi(r)$ is nonincreasing on $[t, r_\star]$ and hence $\phi_\xi(t) \geq \phi_\xi(r_\star)$, and again inequality~\eqref{eqn:key-inequality-on-centered-support} holds.
Markov's inequality implies
\[
\P\{\hat r \leq t\} 
\leq 
\P\Big\{(Z_t-\E Z_t)-(Z_{r_\star}-\E Z_{r_\star}) 
\geq \frac{(t- r_\star)^2}{2} \Big\}
\leq \frac{1}{\psi\Big(\frac{1}{2}\frac{(t-r_\star)^2}{\kappa(t+r_\star)}\Big)}.
\]
We can write $t = r_\star \pm s$ in the above cases. We obtain 
\[
\P\{|\hat r - r_\star| > s\} 
\leq \frac{2}{\psi\Big(\frac{1}{2}\frac{s^2}{\kappa (2 r_\star + s)}\Big)} \leq \frac{2}{\psi\Big(\frac{1}{8}\frac{s^2}{\kappa \max\{r_\star, s\}}\Big)}. 
\]
Take $s = A_{\psi}(C, \mu) s' = \max\{\kappa, \sqrt{r_\star \kappa}\} s'$ for some $s' \geq 1$. Then 
\[
\frac{s^2}{\kappa \max\{r_\star, s\}}
= \frac{\max\{\kappa, r_\star \}}{\max\{r_\star, \kappa s', \sqrt{r_\star \kappa} s'\}} (s')^2 \geq s'.
\]
Hence, 
\[
\P\{|\hat r - r_\star| > A_{\psi}(C, \mu) s'\} \leq \frac{2}{\psi(s'/8)},
\]
as required; we may take $c_1 = 2, c_2 = 1/8$. 
\end{proof}

We are now in a position to complete the proof of~\Cref{thm:distance-from-typical-radius}. 

\begin{proof}[Proof of~\Cref{thm:distance-from-typical-radius}]
From the tail bounds in~\Cref{lem:tail-bounds-for-projection-norm}, note that if $\psi(t) = t^2$, then $\kappa = K_\psi(\mu) = \CvxLip{\mu}$. We have, in this case, 
\[
\P\Big\{|\hat r - r_\star| > \max\{\kappa, \sqrt{r_\star \kappa}\} 
s\Big\}
\lesssim \frac{1}{s^2}, \quad s \geq 1.
\]
In particular, we obtain
\begin{align}
\E |\hat r - r_\star| &= \int_0^\infty 
\P\{|\hat r - r_\star| > s\} \, \ud s \nonumber \\  
&\lesssim  
\max\{\kappa, \sqrt{r_\star \kappa}\} + 
\max\{\kappa, \sqrt{r_\star \kappa}\} \int_{1}^\infty  \frac{1}{s^2} \, \ud s  \nonumber \\ 
&\asymp \max\{\kappa, \sqrt{r_\star \kappa}\}.
\label{ineq:bound-on-expected-difference}
\end{align} 
Consequently, from the bound~\eqref{ineq:bound-on-expected-difference} and~\Cref{prop:upper-bound-via-typical-radius}\ref{item:upper-bound-on-mean-norm}, we have 
\[
\E |\hat r-r_\star|
\lesssim 
\min\Big\{r_\star, \max\{\kappa, \sqrt{r_\star \kappa}\}\Big\} 
\asymp 
\min\{r_\star, \sqrt{r_\star \kappa}\},
\]
as claimed.
\end{proof}

\section{Additional remarks}
\label{sec:additional-remarks}

In this section, we discuss our main results in greater detail.

\begin{remark}[Near extremality of supporting vectors]
The proof of the lower bounds in~\Cref{thm:optimal-reverse-holder-sphere,thm:optimal-reverse-holder-Gauss} for $q \in (p, p \log(\e n/p)]$ made use of a 
reduction to the supporting vectors of convex sets. It turns out that for $q$ much larger than $p$, the supporting vectors are insufficient to match the upper bound by a logarithmic factor in the dimension $n$. 
Indeed, for $1 \leq p \leq q \leq n$, 
from~\Cref{thm:optimal-reverse-holder-sphere} and~\Cref{prop:extremality-of-supporting-vectors} (given below),  
\[
\frac{\ExtProbProj{p}{q}{\sigma_{n-1}}}{\ExtProbSuppVec{p}{q}{\sigma_{n-1}}} 
\asymp \bigg(\frac{\log(\e n/p)}{\log(\e n/q)}\bigg)^{\tfrac{1}{2}(1 - \tfrac{p}{q})} \lesssim \sqrt{\log n}.
\]
(The final inequality is attained with $p = 1$ and $q = n$.)
In other words, up to a logarithmic factor in the dimension, the lack of dimension-free reverse Hölder inequalities for projections can be explained by the heavy-tailed nature of the supporting vectors themselves. 
\end{remark}

\begin{proposition}[Solution to extremal problem for supporting vectors]
\label{prop:extremality-of-supporting-vectors}
For $n \geq 1$ and $1 \leq p < q \leq \infty$, it holds that
\[
\ExtProbSuppVec{p}{q}{\gamma_{n}} 
=
\ExtProbSuppVec{p}{q}{\sigma_{n-1}} 
\simeq
\bigg(\frac{n}{(\twomin{p}{n})\log(\e \tfrac{n}{\twomin{p}{n}})}\bigg)^\tau,
\quad \mbox{where}~~\tau = \frac{1}{2}\Big(1 -
\frac{p}{q}\Big).
\]
\end{proposition}
\begin{proof}
The first relation follows from the $0$-positive homogeneity of the supporting vector map $x\mapsto \|\SuppVec{C}\|_2$. Recalling the limit relation~\eqref{eqn:limit-relation-for-supporting-vector} for the supporting vectors, we obtain for any $C \in \AllConvSets{n}$ with $C \subset B^n_2$, and any $1 \leq p < \infty$, 
\[
\|\SuppVec{C}\|_{L^p(\sigma_{n-1})} = 
\lim_{\alpha \to 0^+} \alpha^{-1} 
\|\Pi_{\alpha C}\|_{L^p(\sigma_{n-1})} 
\gtrsim \rad_2(C) \,  
\sqrt{\frac{(\twomin{p}{n}) \log(\tfrac{\e n}{\twomin{p}{n}})}{n}}.
\]   
The final inequality follows by combining 
the bounds 
from~\Cref{thm:radius-p-moment} for $p \leq n$ and from~\Cref{lem:support-lower-bound} for $p \geq n$.
Note that $\SuppVec{\alpha C}(\theta)=\alpha\, \SuppVec{C}(\theta)$ for any
$\alpha>0$. The case $C=\{0\}$ is trivial; otherwise, for compact
$C\in\AllConvSets{n}$, set $C'=\rad_2(C)^{-1}C$.
Then $C'\subset B_2^n$, $\rad_2(C')=1$, and
$\SuppVec{C'}(\theta)=\rad_2(C)^{-1}\SuppVec{C}(\theta)$. Applying the inequality above
to $C'$ gives
\begin{align}
\|\SuppVec{C}\|_{L^p(\sigma_{n-1})}
&= \rad_2(C)\,\|\SuppVec{C'}\|_{L^p(\sigma_{n-1})} \nonumber \\
&\gtrsim \rad_2(C)
\sqrt{\frac{(\twomin{p}{n}) \log(\tfrac{\e n}{\twomin{p}{n}})}{n}}
\nonumber \\
&\geq \|\SuppVec{C}\|_{L^\infty(\sigma_{n-1})}
\sqrt{\frac{(\twomin{p}{n}) \log(\tfrac{\e n}{\twomin{p}{n}})}{n}}.
\label{ineq:lower-bound-on-support-l-inf}
\end{align}
The final inequality used
$\|\SuppVec{C}(\theta)\|_2\leq \rad_2(C)$ for all $\theta \in \bfS^{n-1}$. 
Hölder interpolation gives for $1 \leq p < q < \infty$, and any compact $C \in \AllConvSets{n}$ with $C \neq \{0\}$, 
\[
\frac{\|\SuppVec{C}\|_{L^q(\sigma_{n-1})}}{\|\SuppVec{C}\|_{L^p(\sigma_{n-1})}}
\leq 
\bigg(\frac{\|\SuppVec{C}\|_{L^\infty(\sigma_{n-1})}}{\|\SuppVec{C}\|_{L^p(\sigma_{n-1})}}\bigg)^{1-p/q} \lesssim \bigg(\frac{n}{(\twomin{p}{n}) \log(\tfrac{\e n}{\twomin{p}{n}})}\bigg)^{\tfrac{1}{2}(1-\tfrac{p}{q})}.
\] 
Note that from inequality~\eqref{ineq:lower-bound-on-support-l-inf}, the same bound continues to hold if $q = \infty$. Passing to the supremum over such sets $C$ and combining with the lower bound from \Cref{prop:auxiliary-problem}\ref{item:lower-bound-on-spherical}, we obtain the result. 
\end{proof}

Based on the proof of the lower bounds, the extremizers
underlying~\Cref{thm:optimal-reverse-holder-sphere,thm:optimal-reverse-holder-Gauss}, up to universal constants, are realized
within the family of convex sets
(see~\Cref{sec:lower-bound-in-the-spherical-case}):
\[
C(a, b) = b \conv(aB^n_2,
e_1)\quad \mbox{for particular choices of}~a, b > 0.
\]
It is tempting to conjecture that
the reason for large $\ExtProbProj{p}{q}{\mu}$ is the ``cone-like'' nature of
these sets. However, this intuition is incorrect.

\begin{remark}[Dimension-free inequalities for conic projections]
Let $\AllConvCones{n}$ denote the class
of
closed convex cones in $\R^n$ and define for 
a probability measure $\mu$ on $\R^n$,
\[
\widetilde{\cA}^\star_{p,q}(\mu) = \sup_{K \in \AllConvCones{n}, K \neq \{0\}}
\frac{\|\Proj{K}\|_{L^q(\mu)}}{\|\Proj{K}\|_{L^p(\mu)}}, \qquad \mbox{for}~1
\leq p < q \leq \infty.
\]
It is straightforward to check that for $K \in \AllConvCones{n}$, the map $x \mapsto
\|\Proj{K}(x)\|_2$ is a
$1$-positively homogeneous convex function on $\R^n$; in fact, this precisely
characterizes when $K \in \AllConvSets{n}$ is a 
cone. Therefore, by the Borell lemma, for an even, log-concave measure $\mu$ on $\R^n$, it follows that 
\begin{equation}
\ExtProbProjCone{p}{q}{\mu} \lesssim \frac{q}{p}, 
\label{eqn:dimension-free-inequality-for-log-concave-measure-and-cones}
\end{equation}
for $1 \leq p \leq q < \infty$.
\end{remark}

The dimension-free nature of the
inequality~\eqref{eqn:dimension-free-inequality-for-log-concave-measure-and-cones} should be 
contrasted with~\Cref{thm:optimal-reverse-holder-Gauss}, where, for Gaussian
measure, the optimal
inequality for general convex sets generally exhibits dimension
dependence. In fact, in the case of spherical or Gaussian measure, it is
possible to give an \emph{exact} solution to the underlying extremal problem
when restricted to convex cones.

\begin{remark}[Exact conic extremizers on the sphere or in Gauss space]
\label{remark:cone}
First, observe that by
$1$-homogeneity and
integration in polar coordinates, for $1 \leq p < q < \infty$ and any $n
\geq 1$, it holds that
\begin{equation}
\label{eqn:rescaling-for-the-sphere}
\ExtProbProjCone{p}{q}{\gamma_n}= \frac{\|\chi_n\|_{L^q}}{\|\chi_n\|_{L^p}}
\,
\ExtProbProjCone{p}{q}{\sigma_{n-1}}\asymp \sqrt\frac{\twomax{q}{n}}{\twomax{p}{n}}\,
\ExtProbProjCone{p}{q}{\sigma_{n-1}},
\end{equation}
by~\Cref{lem:chi-n-concentration}. Moreover, the relation above shows that the extremizers for any triple $(n, p, q)$ coincide on the sphere and in Gauss space. In fact, as shown below in~\Cref{thm:extremal-problem-cones}, we have 
\[
\ExtProbProjCone{p}{q}{\gamma_n} \asymp \sqrt{\frac{q}{p}}, \quad \mbox{and}
\quad
\ExtProbProjCone{p}{q}{\sigma_{n-1}} \asymp \sqrt{\frac{q/p}{(\twomax{q}{n})/(\twomax{p}{n})}},
\]
where a one-dimensional ray is exactly extremal.
\end{remark}

\begin{theorem}[Exact solution to the extremal problem for cones]
\label{thm:extremal-problem-cones}
Let $n\geq 1$ and $1\leq p<q<\infty$. Then
\[
\ExtProbProjCone{p}{q}{\gamma_n}
=
2^{1/p-1/q}
\frac{\|Z\|_{L^q}}{\|Z\|_{L^p}},
\]
where $Z\sim N(0,1)$. Consequently,
\[
\ExtProbProjCone{p}{q}{\sigma_{n-1}}
=
2^{1/p-1/q}
\frac{\|Z\|_{L^q}/\|Z\|_{L^p}}
{\|\chi_n\|_{L^q}/\|\chi_n\|_{L^p}}.
\]
Moreover, both suprema are attained by the ray
\[
K_n^\star=\{\lambda e_1:\lambda\geq 0\}.
\]
\end{theorem}

Before giving the proof, we note that a limiting argument applied to~\Cref{thm:extremal-problem-cones} yields 
\[
\ExtProbProjCone{p}{\infty}{\sigma_{n-1}} = 2^{1/p}
\frac{\|\chi_n\|_{L^p}}{\|Z\|_{L^p}} \asymp 
\begin{cases}
\sqrt{n/p}, & \mbox{if}~p \leq n \\ 
1, & \mbox{if}~p \geq n
\end{cases},
\]
if $1 \leq p < \infty$. The last relation above follows from~\Cref{lem:chi-n-concentration}.

\begin{proof}
We first prove the result in Gauss space.
Let $K\in \AllConvCones{n}$, $K\neq\{0\}$. 
Take $C=K\cap B_2^n$.
For every $x\in \R^n$, Moreau decomposition gives
\begin{equation}
\label{eqn:projection-formula}    
h_C(x)
=
\sup_{y\in K\cap B_2^n}\langle x,y\rangle
=
\|\Proj{K}(x)\|_2.
\end{equation}
Let $\check C_p$, for $p \geq 1$, denote the $p$th-symmetral, \ie the centrally symmetric convex set whose support function satisfies
\[
h_{\check C_p}(x) = \Big(\frac{h_{C}(x)^p + h_{C}(-x)^p}{2}\Big)^{1/p}, 
\quad \mbox{for any}~x \in \R^n.
\]
Note that $a^r + b^r \leq (a^p + b^p)^{r/p}$ for any $r \geq p \geq 1$ and
any $a, b \geq 0$, with equality if $r = p$. This implies that with $G \sim \gamma_n$,
\begin{align*}
\E h^r_{C}(G) &= \half \E \big[h^r_C(-G) + h^r_C(G)\big] \\ 
&\leq \half \E \Big[ \big(h^p_C(-G) + h^p_C(G)\Big)^{r/p} \Big]
= 2^{r/p - 1}\E h_{\check C_p}^r(G), 
\end{align*}
with equality if $r = p$. Equivalently, for every $r \geq p \geq 1$,
\[
\|h_C\|_{L^r(\gamma_n)} \leq 2^{1/p - 1/r} \|h_{\check C_p}\|_{L^r(\gamma_n)},
\]
with equality if $r = p$. 
Hence, for $q > p \geq 1$ with $q$ finite, 
\begin{subequations}
\begin{equation}
\label{eqn:upper-inequality-for-a-cone}
\frac{\|\Proj{K}\|_{L^q(\gamma_n)}}{\|\Proj{K}\|_{L^p(\gamma_n)}}
= 
\frac{\|h_C\|_{L^q(\gamma_n)}}{\|h_C\|_{L^p(\gamma_n)}}
\leq 2^{1/p - 1/q}
\frac{\|h_{\check C_p}\|_{L^q(\gamma_n)}}{\|h_{\check C_p}\|_{L^p(\gamma_n)}}
\leq 
2^{1/p - 1/q}
\frac{\|Z\|_{L^q}}{\|Z\|_{L^p}},
\end{equation}
by the sharp form of the Gaussian Kahane--Khinchine inequality (due to
Lata{\l}a-Oleszkiewicz~\cite[Corollary 3]{LatOle99}, who attribute
the
result to S.\ Szarek).
If $G \sim \gamma_n$ and $Z \sim \gamma_1$, then
\[
\|\Proj{K^\star_n}\|_{L^r(\gamma_n)} = \|(G_1)_+\|_{L^r} = 2^{-1/r}
\|Z\|_{L^r}, \quad \mbox{for}~1 \leq r < \infty.
\]
Therefore, 
\begin{equation}
\label{eqn:lower-inequality-for-a-cone}
\frac{\|\Proj{K^\star_n}\|_{L^q(\gamma_n)}}{\|\Proj{K^\star_n}\|_{L^p(\gamma_n)}}
= 2^{1/p - 1/q}
\frac{\|Z\|_{L^q}}{\|Z\|_{L^p}}.
\end{equation}
\end{subequations}
Thus, the result follows, for Gauss space, by combining the inequalities 
\eqref{eqn:upper-inequality-for-a-cone} and 
\eqref{eqn:lower-inequality-for-a-cone}.
The spherical result similarly follows from 
the Gaussian
result and the rescaling
relation~\eqref{eqn:rescaling-for-the-sphere}.
\end{proof}

\begin{remark}[Statistical consequences]
Consider the statistical minimax risk
\[
\mathcal{M}_p(C, \mu) = \inf_{\hat \theta} \sup_{\theta \in C} \Bigg(\E_{\xi \sim \mu} \Big[\|\hat \theta(\theta + \xi) - \theta\|_2^p\Big]\Bigg)^{1/p},
\quad \mbox{for}~p\geq 1.
\]
It can be verified that $\mathcal M_1(C, \gamma_n) \asymp \mathcal M_2(C, \gamma_n)$ for any closed convex set $C \subset \R^n$. Let $C^\star_n$ be approximately extremal, \ie such that 
\[
\|\Proj{C^\star_n}\|_{L^2(\gamma_n)} \asymp \ExtProbProj{1}{2}{\gamma_n} \, 
\|\Proj{C^\star_n}\|_{L^1(\gamma_n)}.
\] 
Since $\ExtProbProj{1}{2}{\gamma_n} \gg 1$, it follows that for $\hat \theta(y) = \Proj{C^\star_n}(y)$,
\[
\sup_{\theta \in C^\star_n} \E_{\xi \sim \gamma_n} \|\hat \theta(\theta + \xi) - \theta\|_2^2 \gg \big[\mathcal M_2(C^\star_n, \gamma_n)\big]^2.
\]
In other words, the least-squares estimator (LSE) is also minimax-suboptimal for $C^\star_n$.
\end{remark}

\bibliographystyle{abbrv}
\bibliography{references}
\end{document}